\documentclass[12pt]{amsart}
\usepackage[english]{babel}
\usepackage{amsfonts,amssymb}
\usepackage{dsfont}
\usepackage{verbatim}
\newtheorem{thm}{Theorem}

\newtheorem{lemma}{Lemma}

\newtheorem{prop}{Proposition}
\theoremstyle{definition}
\newtheorem{defn}{Definition}

\newcommand{\set}[1]{\left\{#1\right\}}
\newcommand{\norm}[1]{\left\Vert#1\right\Vert}
\newcommand{\ov}[1]{\overline{#1}}
\newcommand{\vre}{\varepsilon}

\newcommand{\cX}{{\mathcal{X}}}
\newcommand{\cA}{{\mathcal{A}}}
\newcommand{\cB}{{\mathcal{B}}}

\newcommand{\cD}{{\mathcal{D}}}

\newcommand{\cF}{{\mathcal{F}}}

\newcommand{\cH}{{\mathcal{H}}}

\newcommand{\cK}{{\mathcal{K}}}
\newcommand{\cL}{{\mathcal{L}}}
\newcommand{\cN}{{\mathcal{N}}}

\newcommand{\cY}{{\mathcal{Y}}}
\newcommand{\cU}{{\mathcal{U}}}

\newcommand{\dC}{{\mathds{C}}}
\newcommand{\dN}{{\mathds{N}}}
\newcommand{\dR}{{\mathds{R}}}
\newcommand{\dT}{{\mathds{T}}}
\newcommand{\dZ}{{\mathds{Z}}}

\newcommand{\gY}{\mathfrak{Y}}
\newcommand{\gX}{\mathfrak{X}}

\newcommand{\gp}{\mathfrak{p}}
\newcommand{\got}{\mathfrak{t}}

\newcommand{\Ran}{\mathrm{Ran}}
\newcommand{\img}{\mathrm{\bf i}}
\newcommand{\supp}{\mathrm{supp}}

\newcommand{\Un}{\mathbf{1}}
\newcommand{\One}{\mathds{1}}
\newcommand{\tr}{\mathrm{tr}}
\newcommand{\muw}{\widetilde{\mu}}

\title{On $q$-normal operators and quantum complex plane.}

\begin{document}
\author{Jaka Cimpri\v c}
\address{University of Ljubljana, Department of Mathematics, Jadranska 19, SI-1000 Ljubljana, Slovenia}
\email{Jaka.Cimpric@fmf.uni-lj.si}

\author{Yurii Savchuk}
\address{Universit\"at Leipzig, Mathematisches Institut, Johannisgasse 26, 04103 Leipzig, Germany}
\email{savchuk@math.uni-leipzig.de}

\author{Konrad Schm\"udgen}
\address{Universit\"at Leipzig, Mathematisches Institut, Johannisgasse 26, 04103 Leipzig, Germany}
\email{schmuedgen@math.uni-leipzig.de}


\date{\today}


\maketitle
\begin{abstract}
For $q>0$ let $\cA$ denote the unital $\ast$-algebra with generator $x$ and defining relation $xx^\ast=qxx^\ast$. Based on this algebra we study $q$-normal operators, the complex $q$-moment problem, positive elements and sums of squares.
\end{abstract}

\section{Introduction}
Suppose that $q$ is a positive real number. A densely defined closed linear operator $X$ on a Hilbert space is called \textit{ $q$-normal} if
\begin{gather}\label{relationxxastq}
XX^*=qX^*X.
\end{gather}
This and other classes of $q$-deformed operators have been introduced and investigated by S. Ota \cite{ota}, see e.g. \cite{osz2}. In this paper we continue the study of $q$-normal operators. Further, let $\cA$ denote the unital complex $\ast$-algebra with single generator $x$ and defining relation
\begin{gather}\label{relationxxastq1}
xx^\ast =q x^\ast x.
\end{gather}
The algebra $\cA$ appears in the theory of quantum groups where it is considered as the coordinate algebra of the $q$-deformed complex plane, or briefly, of the \textit{complex $q$-plane.} 

Let us set for a moment $q=1$. Then the $q$-normal operators are precisely the normal operators and $\cA$ is a complex polynomial algebra $\dC[x,\ov{x}]$. It is well known that there is a close relationship between various important algebraic and analytic problems:
\begin{itemize}
\item the complex moment problem,
\item the extension of formally normal operators to normal operators (possibly in a larger Hilbert space),
\item the characterization of well-behaved representations of the $\ast$-algebra $\dC[x,\overline{x}]$,
\item the representation of positive polynomials as sums of squares (motivated by 17-th Hilbert problem)
\end{itemize}
The aim of the present paper is to begin a study of these problems and their interplay in the $q$-deformed case.

We now discuss the contents of this paper. Section \ref{q-Normal operators} deals with $q$-normal operators. After giving some equivalent characterizations of $q$-normality we prove a structure theorem for $q$-normal operators (Theorem \ref{thm_spec}) which is the counter-part of the spectral theorem for unbounded normal operators.

Section \ref{positive} deals with positive $q$-polynomials and their possible representations as sums of squares. We define the cone $\cA_+$ of \textit{ positive} elements to be the set of elements which are mapped into positive symmetric operators by all well-behaved $\ast$-representations of the $\ast$-algebra $\cA$. A $\ast$-representation $\pi$ of $\cA$ with domain $\cD(\pi)$ is called \textit{well-behaved} if there exists a $q$-normal operator $X$ such that $\cD(\pi)=\cap_{n=1}^\infty \cD(X^n)$ and $\pi(x)=X \lceil \cD(\pi)$.
Theorem \ref{thm_motzkin1} states that for each positive $q{\neq} 1$ there exists a polynomial $p_q\in \dR[t]$ of degree four such that the element $f:=p_q(x+x^*)$ is in $\cA_+$, but $f$ is not a sum of squares in $\cA$. In contrast we prove that if $p\in\dR[t]$ and $p(x^*x)\in\cA_+$, then $p(x^*x)$ is always a sum of squares.

In Section \ref{q-momentproblem} we study a generalization of the complex moment problem to the $\ast$-algebra $\cA$. Let $F$ be a linear functional on $\cA$. We say that $F$ is a \textit{$q$-moment functional} if there exists a well-behaved $\ast$-representation $\pi$ of $\cA$ and a vector $\varphi \in \cD(\pi)$ such that $F(a)=\langle \pi(a)\varphi,\varphi\rangle$ for $ a\in\cA$. (In Section \ref{q-momentproblem} we use Theorem \ref{thm_spec} to give a formulation of $q$-moment functionals in terms of measures.) Further, $F$ is called \textit{ positive} if $F(f^\ast f)\geq 0$ for all $f\in \cA$ and \textit{strongly positive} if $F(f)\geq0$ for all $f\in \cA_+$. Then, by Theorem \ref{thm_hav}, a linear functional on $\cA$ is a $q$-moment functional if and only if it is strongly positive. This result can be considered as the counter-part of Haviland's theorem. In contrast, by Theorems \ref{notmomentfunctional} and \ref{prop_noext}, there exists a positive linear functional $F$ on $\cA$ which is not a $q$-moment functional and a formally $q$-normal operator which has no $q$-normal extensions.

In the commutative case $q=1$ the solution of Hilbert's 17-th problem (see e.g. \cite{m}) implies that positive polynomials of $\dC[x,\ov{x}]$ are sums of squares of rational functions. In
Section \ref{strictpositivstellensatz} we prove a Positivstellensatz (Theorem \ref{thm_strpos}) which states, roughly speaking, that \textit{strictly positive} elements of $\cA_+$ can be represented as sums of squares by allowing "nice" denominators.

We close this introduction by collecting some definitions and notations.

\noindent \textbf{Operators in Hilbert space.} For a linear operator $A$ on a Hilbert space operator we denote by $\cD(A),\ \Ran A,\ \overline{A}$ and $A^*$ denote its \emph{domain}, its \emph{range}, its \emph{closure} and its \emph{adjoint}, respectively, and we set $\cD^\infty(A):=\cap_{n=1}\cD(A^n).$ A \emph{core} of a closed operator $A$ is a linear subset $\cD_0\subseteq\cD(A)$ such that the closure of $A\upharpoonright\cD_0$ coincides with $A.$ If $A$ is self-adjoint, then $\cD^\infty(A)$ is a core of $A.$

A number $\lambda\in\dC$ is called a \textit{regular point} for an operator $A$ if there exists $c_\lambda >0$ such that
$\norm{(A-\lambda I)\varphi}\geq c_\lambda \norm{\varphi}$ for all $\varphi\in\cD(A).$

If $A_i$, $i\in I$, are linear operators on a Hilbert space $\cH_i,$ the \textit{direct sum} $\oplus_{i\in I} A_i$ denotes the operator on $\cH:=\oplus_{i\in I} \cH_i$ defined by $(\oplus_{i\in I} A_i)(\varphi_i)_{i\in I}:=(A_i\varphi_i)_{i\in I}$ for $(\varphi)_{i\in I}$ in\\

$\cD(\oplus_{i\in I}A_i) :=\{(\varphi_i)_{i\in I}: \varphi_i\in\cD(A_i),(\varphi_i)_{i\in I}\in \cH ~{\rm and}~(A_i\varphi_i)_{i\in I}\in \cH\}.$\\

\noindent \textbf{$*$-Algebras and $*$-representations.} By a $*$-\textit{algebra} we mean a complex associative algebra $\cA$ equipped with a mapping $a\mapsto a^*$ of $\cA$ into itself, called the \textit{involution} of $\cA$, such that $(\lambda a+\mu b)^* = \bar{\lambda}a^*+ \bar{\mu} b^*, (ab)^* = b^* a^*$ and $(a^*)^*=a$ for $a,b\in \cA$ and $\lambda, \mu\in \dC$. In this paper each $*$-algebra $\cA$ has an identity element denoted by $\Un_\cA$ or $\Un.$

An element of the form $\sum^n_{j=1} a_j^*a_j^{}$, where $ a_1,\dots,a_n\in\cA$ is called a \textit{sum of squares} in $\cA.$ The set of all sums of squares is denoted by $\sum\cA^2.$

We use some terminology and results from unbounded representation theory in Hilbert space (see e.g. in \cite{s4}). Let $\cD$ be a dense linear subspace of a Hilbert space $\cH$ with scalar product $\langle\cdot,\cdot\rangle.$ A $*$-\textit{representation} of a $*$-algebra $\cA$ on $\cD$ is an algebra homomorphism $\pi$ of $\cA$ into the algebra $\cL(\cD)$ of linear operators on $\cD$ such that $\pi(\Un)=I_\cD$ and $\langle\pi(a)\varphi,\psi\rangle=\langle\varphi,\pi(a^*)\psi\rangle$ for all $\varphi,\psi\in\cD$ and $a\in \cA.$ We call $\cD(\pi):=\cD$ the \textit{domain} of $\pi$ and write $\cH(\pi):=\cH.$ A $*$-representation is \emph{faithful} if $\pi(a)=0$ implies $a=0.$


We say that an element $a=a^*\in\cA$ is \emph{positive} in a $*$-representation $\pi$ if $\langle\pi(a)\varphi,\varphi\rangle\geq 0$ for all $\varphi\in\cD(\pi).$

Suppose that $\pi$ is a $*$-representation of $\cA$. The \textit{graph topology} of $\pi$ is the locally convex topology on the vector space $\cD(\pi)$ defined by the norms $\varphi\mapsto\norm{\varphi}+\norm{\pi(a)\varphi},$ where $a\in\cA$. Then $\pi$ is \textit{closed} if and only if $\cD(\pi)=\cap_{a\in\cA}\cD(\overline{\pi(a)}).$ We say that $\pi$ is \textit{strongly cyclic} if there exists a vector $\xi\in\cD(\pi)$ such that $\pi(\cA)\xi$ is dense in $\cD(\pi)$ in the graph topology of $\pi$. 

A linear functional $F:\cA\to\dC$ on a $*$-algebra $\cA$ is \emph{positive} if $F(\sum\cA^2)\geq 0.$ Every positive functional $F$ has a GNS representation (see e.g. \cite{s4}), that is, there exists a $*$-representation $\pi_F$ and with cyclic vector $\varphi$ such that $F(a)=\langle\pi_F(a)\varphi,\varphi\rangle$ for $a\in\cA.$


\section{$q$-Normal operators}\label{q-Normal operators}
In what follows $q$ is a positive real number. Recall the definition of a $q$-normal operator, see e.g. \cite{ota}.

\begin{defn}
A densely defined operator $X$ on a Hilbert space $\cH$ is a $q$-normal operator if $\cD(X)=\cD(X^*)$ and $\norm{X^*f}=\sqrt{q}\norm{Xf},\ f\in\cD(X).$
\end{defn}

A $q$-normal operator with $q=1$ is normal. Each $q$-normal operator $X$ is closed. Indeed, since $\norm{X^*f}=\sqrt{q}\norm{Xf},$ the graph norms of $X$ and $X^*$ are equivalent. Therefore, since $X^*$ is closed, $X$ is also closed. It also implies that $\ker X=\ker X^*.$ 

The following proposition collects different characterizing properties of $q$-normal operators, cf. Chapter 2 in \cite{osam}.
\begin{prop}\label{prop_char}Let $X$ be a closed operator on a Hilbert space $\cH$ and let $X=UC$ be its polar decomposition. The following statements are equivalent:
\begin{enumerate}
  \item[$(i)$]\quad $X$ is $q$-normal,
  \item[$(ii)$]\quad $XX^*=qX^*X,$
  \item[$(iii)$]\quad $UC^2U^*=qC^2,$
  \item[$(iv)$]\quad $UCU^*=q^{1/2}C,$
  \item[$(v)$]\quad $U E_C(\Delta)U^*=E_C(q^{-1/2}\Delta)$ for each Borel $\Delta\subseteq\dR_+,$
  \item[$(vi)$]\quad $U f(C)U^*=f(q^{1/2}C)$ for every Borel function $f$ on $\dC.$
\end{enumerate}
\end{prop}
\begin{proof}


\noindent $(i)\Rightarrow(ii):$ We use some basic properties of quadratic forms associated with positive operators, see e.g. \cite{rs}. Introduce the quadratic forms $\got_1[\varphi,\varphi]:=\langle X^*\varphi,X^*\varphi\rangle$ and $\got_2[\varphi,\varphi]:=q\langle X\varphi,X\varphi\rangle,\ \cD[\got_1]=\cD(X^*),\ \cD[\got_2]=\cD(X).$ Let $X$ be $q$-normal. Then $X$ is closed and
$$
\got_1[\varphi,\varphi]=\norm{X^*\varphi}^2=q\norm{X\varphi}^2=\got_2[\varphi,\varphi].
$$
Thus, we get $\got_1=\got_2.$ Since $X$ and $X^*$ are closed, $\got_1$ and $\got_2$ are closed. The operators associated with $\got_1$ and $\got_2$ are $XX^*$ and $qX^*X$ respectively. Hence $XX^*=qX^*X.$

\medskip
\noindent $(ii)\Rightarrow(iii):$ Since $X=UC$ is a polar decomposition, we have $X^*=CU^*$ and $\ker U=\ker C.$ Then equation $(ii)$ implies $UC^2U^*=qCU^*UC=qC^2.$

\medskip
\noindent $(iii)\Rightarrow(iv):$ Equation $(iii)$ implies that for $\varphi\in\cD(C^2)$ holds
\begin{gather*}
q\norm{C\varphi}^2=q\langle C^2\varphi,\varphi\rangle=\langle UC^2U^*\varphi,\varphi\rangle=\langle CU^*\varphi,CU^*\varphi\rangle=\norm{CU^*\varphi}^2.
\end{gather*}

It implies $\ker U^*\subseteq\ker C.$ On the other hand, if $\varphi\in\ker C,$ then by the last equation $U^*\varphi\in\ker C=\ker U.$ That is $UU^*\varphi=0,$ which implies $U^*\varphi=0.$ Hence $\ker U^*=\ker C=\ker U.$ Restricting $U,U^*,C$ onto $(\ker C)^\perp$ we can assume that $U$ is unitary. Then relation $(iii)$ defines a unitary equivalence of $C^2$ and $qC^2.$ Hence, the square roots $C$ and $q^{1/2}C$ are unitary equivalent and we get $(iv).$

\medskip
\noindent $(iv)\Rightarrow(v):$ As in the previous case we can assume that $U$ is unitary. Then $C$ and $q^{1/2}C$ are unitary equivalent and for every Borel $\Delta\subseteq\dR$ we get
\begin{gather*}
U E_C(\Delta)U^*=E_{q^{1/2}C}(\Delta)=E_C(q^{-1/2}\Delta).
\end{gather*}

\medskip
\noindent $(v)\Rightarrow(i):$ Note that $(v)$ implies that $U$ and $U^*$ commute with $E_C(\set{0}),$ that is $\ker C$ is invariant under $U$ and $U^*.$ Considering the restriction of $X$ (resp. $U$ and $C$) onto $(\ker C)^\perp$ we can assume that $U$ is unitary. Then $(v)$ means that $C$ and $q^{1/2}C$ are unitarily equivalent, namely $UCU^*=q^{1/2}C,$ which implies $CU^*=q^{1/2}U^*C.$ Using the latter we get $\cD(X)=\cD(UC)=\cD(U^*C)=\cD(CU^*)=\cD(X^*)$ and
$$
\langle CU^*\varphi,CU^*\varphi\rangle=q\langle U^*C\varphi,U^*C\varphi\rangle=q\langle UC\varphi,UC\varphi\rangle,\ \varphi\in\cD(X)
$$
which implies $\norm{X^*\varphi}=q^{1/2}\norm{X\varphi}.$

\medskip
\noindent $(v)\Rightarrow(vi):$ Follows from the computation
\begin{gather*}
Uf(C)U^*=U\left(\int f(\lambda)d E_C(\lambda)\right) U^*=\int f(\lambda)d E_C(q^{-1/2}\lambda)=\int f(\lambda)d E_{q^{1/2}C}(\lambda)=f(q^{1/2}C).
\end{gather*}

\medskip
\noindent $(vi)\Rightarrow(v):$ Follows by setting $f(\lambda)=\One_{\Delta}(\lambda).$
\end{proof}


\bigskip
We provide a basic example of a $q$-normal operator for $q\neq 1.$ Put
\begin{gather}\label{eq_Delta_q}
\Delta_q=\left\{
           \begin{array}{ll}
             [1,q^{1/2}), & \hbox{if}\ q>1, \\
             (q^{1/2},1], & \hbox{if}\ q<1.
           \end{array}
         \right.
\end{gather}

\noindent Let $\mu$ be a Borel measure on $\dR_+=[0,+\infty)$ such that
\begin{gather}\label{eq_muD}
\mu(\Delta)=\mu(q^{1/2}\Delta)\ \mbox{for all Borel subsets}\ \Delta\subseteq\dR_+.
\end{gather}
Since $(0,+\infty)=\cup_{k\in\dZ}q^k\Delta_q,$ the measure $\mu$ is uniquely defined by its restriction onto the subspace $\Delta_q\cup\set{0}\subset\dR_+.$ Define an operator $X_\mu,$ on the Hilbert space $\cH_\mu:=L_2(\dR_+,d\mu)$ as follows.
\begin{gather}\label{eq_Amu}
(X_\mu\varphi)(t):=q^{1/2}t\varphi(q^{1/2}t),\ \cD(X_\mu)=\set{\varphi(t)\in \cH_\mu|\ t\varphi(t)\in \cH_\mu}.
\end{gather}
Let $\cH_{\mu,0}=\set{f\in\cH_\mu|\ f\equiv 0,\ \mbox{a.e. on}\ (0,+\infty)}.$ Then $\mu(\set{0})=0$ if and only if $\cH_{\mu,0}=\nolinebreak\set{0}.$ We prove the following
\begin{prop}\label{prop_Amu}$ $
\begin{enumerate}
	\item[$(i)$] The operator $X_\mu$ in (\ref{eq_Amu}) is a well-defined $q$-normal operator with $\ker X_\mu=\cH_{\mu,0}.$
	\item[$(ii)$] The adjoint operator $X_\mu^*$ is defined by $$(X_\mu^*\varphi)(t)=t\varphi(q^{-1/2}t),\ \cD(X_\mu^*)=\cD(X_\mu).$$
    \item[$(iii)$] Let $X_\mu=U_\mu C_\mu$ be the polar decomposition of $X_\mu.$ Then
    \begin{gather}\label{eq_Umu_Cmu}
(C_\mu\psi)(t):=t\psi(t),\ (U_\mu\varphi)(t):=\left\{
                     \begin{array}{ll}
                       \varphi(q^{1/2}t), & \hbox{for}\ t\neq 0, \\
                       0, & \hbox{for}\ t=0.
                     \end{array}
                   \right.
\end{gather}
where $\varphi\in\cH, \psi\in\cD(C_\mu)=\cD(X_\mu).$
\end{enumerate}
\end{prop}
\begin{proof}
It follows from (\ref{eq_muD}) that $U_\mu$ is a well-defined partial isometry. Further, we have $X_\mu=U_\mu C_\mu$ and $\ker C_\mu=\ker U_\mu=\cH_{\mu,0}.$ Since $C_\mu$ is a positive self-adjoint operator, $X_\mu=U_\mu C_\mu$ is a polar decomposition of $X_\mu.$ Since $U_\mu$ is bounded, we have $X_\mu^*=C_\mu U_\mu^*.$ Together with (\ref{eq_muD}) it implies $\cD(X_\mu^*)=\cD(C_\mu)=\cD(X_\mu).$ For $\varphi\in\cD(X_\mu)$ we calculate using (\ref{eq_muD}):
$$
\norm{X_\mu\varphi}=\norm{t\varphi(t)}\ \mbox{and}\ \norm{X_\mu^*\varphi}=q^{1/2}\norm{t\varphi(t)}.
$$
Thus, $X_\mu$ is $q$-normal.
\end{proof}


The following theorem can be viewed as an analogue of the spectral theorem for normal operators (see e.g. Theorem VII.3 in \cite{rs}).
\begin{thm}\label{thm_spec}
Let $X$ be a $q$-normal operator on a Hilbert space $\cH.$ Then there exists a family of Borel measures $\mu_i,\ i\in I$ on $\dR_+$ satisfying $\mu_i(x)=\mu_i(q^{1/2}x)$ such that $X$ is unitarily equivalent to the direct sum of operators $X_{\mu_i}$ defined by (\ref{eq_Amu}).
\end{thm}
\begin{proof}
Let $\cH_0\subseteq \cH$ denote $\ker X=\ker X^*.$ The restriction $X\upharpoonright\cH_0$ is a direct sum of copies of $X_{\mu_0}$ where $\mu_0(\dR_+)=\mu(\set{0})=1.$ The restriction $X_1=X\upharpoonright\cH_0^\perp$ is again a $q$-normal operator with $X_1^*=X^*\upharpoonright\cH_0^\perp.$ Without loss of generality, we can assume that $\ker X=\ker X^*=\set{0}.$ Let $X=UC$ be the polar decomposition of $X.$ Since $\ker X=\set{0},\ U$ is unitary.

Let $\cK=\Ran E_C(\Delta_q).$ Then $\cK$ is invariant under $C$ and we denote by $D$ the restriction $C\upharpoonright\cK.$ Further, let $\cK=\oplus_{i\in I}\cK_i$ an arbitrary orthogonal sum decomposition such that $\cK_i$ is invariant under $D$ and let $D_i=D\upharpoonright\cK_i,\ i\in I.$ We show that there is a corresponding orthogonal sum decomposition of $X=\oplus_{i\in I}X_i.$

It follows from Proposition \ref{prop_char},$(v)$ that
\begin{gather}\label{eq_U_d_q}
U^k\left(\Ran E_C(\Delta_q)\right)=\Ran E_C(q^{-k/2}\Delta_q),\ k\in\dZ.
\end{gather}

Since $(0,+\infty)$ is a disjoint union of $q^{k/2}\Delta_q,\ k\in\dZ$, we get the following direct sum decomposition
\begin{gather*}
\cH=\Ran E_C((0,+\infty))=\bigoplus_{k\in\dZ}\Ran E_C(q^{k/2}\Delta_q)=\\
=\bigoplus_{k\in\dZ}U^{*k}\Ran E_C(\Delta_q)=\bigoplus_{k\in\dZ}\bigoplus_{i\in I}U^{*k}\cK_i=\bigoplus_{i\in I}\bigoplus_{k\in\dZ}\cK_{i,k}=\bigoplus_{i\in I}\cH_i,
\end{gather*}
where $\cK_{i,k}=U^{*k}\cK_i$ and $\cH_i=\bigoplus_{k\in\dZ}\cK_{i,k}.$ For $i\in I, k\in\dZ$ we define the operators
\begin{gather}\label{eq_D_ik}
D_{i,k}=U^{*k}(q^{k/2}D_i)U^k:\cK_{i,k}\to \cK_{i,k}.
\end{gather}
Then for each $\varphi\in\cK_{i,k}$ holds $U^k\varphi\in\cK_i$ and we calculate using Proposition \ref{prop_char}, $(iv)$
$$
D_{i,k}\varphi=q^{k/2}U^{*k}D_i(U^k\varphi)=q^{k/2}U^{*k}C(U^k\varphi)=CU^{*k}(U^k\varphi)=C\varphi.
$$
It implies that $\cK_{i,k}$ is invariant under $C$ and $C\upharpoonright\cK_{i,k}=D_{i,k}$ for all $i\in I,\ k\in\dZ.$ Put $C_i=\bigoplus_{k\in\dZ}D_{i,k},\ i\in I.$ Then $C_i,\ i\in I$ are self-adjoint with $\Ran C_i\subseteq\cH_i$ and $\bigoplus_{i\in I}C_i=C$ in particular, $C_i=C\upharpoonright\cH_i.$ The subspaces $\cH_i,\ i\in I$ are invariant under $U$ by definition of $\cH_i.$ We denote by $U_i$ the restriction of $U$ onto $\cH_i,$ so that $U=\oplus_{i\in I}U_i.$ Thus we get $X=UC=\oplus_{i\in I}U_iC_i.$

Using Zorn's Lemma we can choose $D_i$ to be cyclic with cyclic vectors $\psi_i\in\cK_i.$ Since $\sigma(D_i)\subseteq\Delta_q,$ there exist unitary operators $V_i:\cK_i\to L_2(\Delta_q,d\mu_i)$ such that
\begin{gather}\label{eq_Vi_Di}
(V_i^{}D_i^{}V_i^*f)(t)=tf(t),\ f(t)\in L_2(\Delta_q,d\mu_i),
\end{gather}
where $\mu_i(\cdot)=\langle E_{D_i}(\cdot)\psi_i,\psi_i\rangle,$ see e.g. Chapter VII in \cite{rs}.

It follows from (\ref{eq_D_ik}) that operators $D_{i,k}$ are cyclic on $\cK_{i,k},$ with cyclic vectors $\psi_{i,k}:=U^{*k}\psi_i,\ i\in I,\ k\in\dZ.$ By (\ref{eq_D_ik}) we also have $\sigma(D_{i,k})\subseteq q^{k/2}\Delta_q.$ We calculate the corresponding measures $\mu_{i,k}(\cdot)=\langle E_{D_{i,k}}(\cdot)\psi_{i,k},\psi_{i,k}\rangle$ using the unitary equivalence (\ref{eq_D_ik}):
\begin{gather}
\nonumber\mu_{i,k}(q^{k/2}\Delta)=\langle E_{D_{i,k}}(q^{k/2}\Delta)U^{*k}\psi_{i},U^{*k}\psi_{i}\rangle=\langle U^{k}E_{D_{i,k}}(q^{k/2}\Delta)U^{*k}\psi_i,\psi_i\rangle=\\
\label{eq_mu_ik}=\langle E_{q^{k/2}D_i}(q^{k/2}\Delta)\psi_i,\psi_i\rangle=\langle E_{D_i}(\Delta)\psi_i,\psi_i\rangle=\mu_i(\Delta),
\end{gather}
for each Borel set $\Delta\subseteq\Delta_q.$ 
For every $k\in\dZ$ we define an operator:
$$
W_k:L_2(\Delta_q,d\mu_i)\to L_2(q^{k/2}\Delta_q,d\mu_{i,k}),\ (Wf)(t)=f(q^{-k/2}t).
$$
It follows from (\ref{eq_mu_ik}) that $W_k$ are unitary. Further, for each $i\in I,\ k\in\dZ$ we define unitary operators $V_{i,k}=W_kV_iU^k,\ V_{i,k}:\cK_{i,k}\to L_2(q^{k/2}\Delta_q,d\mu_{i,k}).$ Equations (\ref{eq_D_ik}),(\ref{eq_Vi_Di}) imply that
$$(V_{i,k}^{}D_{i,k}^{}V_{i,k}^*f)(t)=tf(t)\ \mbox{for}\ f\in L_2(q^{k/2}\Delta_q,d\mu_{i,k}).$$
Since $\supp\mu_{i,k}\subseteq q^{k/2}\Delta_q$ are disjoint for different $k\in\dZ,$ we can define a Borel measure $\muw_i:=\sum_{k\in\dZ}\mu_{i,k}$ on $\dR_+$ which satisfies (\ref{eq_muD}). Then $\widetilde{V_i}=\oplus_{k\in\dZ}V_{i,k}$ is a unitary operator from $\cH_i$ to $L_2(\dR_+,\muw_i)$ such that $\widetilde{V_i}C_i\widetilde{V_i}^*=C_{\muw_i},$ where $C_{\muw_i},\ i\in I$ are defined by (\ref{eq_Umu_Cmu}). Using (\ref{eq_D_ik}) and (\ref{eq_Vi_Di}) we get $\widetilde{V_i}U_i\widetilde{V_i}^*=U_{\muw_i},$ where $U_{\muw_i},\ i\in I$ are defined by (\ref{eq_Umu_Cmu}). Hence, $X=\oplus_{i\in I}X_i,$ where every $X_i=U_iC_i$ is unitary equivalent to $X_{\muw_i},\ i\in I.$
\end{proof}

\begin{defn}
We say that a $q$-normal operator $X$ is \emph{reducible} if $\cD(X)=\cD_1\oplus \cD_2,\ \cD_i\neq 0$ and $\cD_1,\cD_2$ are invariant under $X.$ Otherwise we say that $X$ is \emph{irreducible}.
\end{defn}

For irreducible $q$-normal operators we obtain the following description, see also \cite{osam}, p.71.
\begin{prop}\label{prop_des_irr}
Let $X$ be a non-zero irreducible $q$-normal operator on a Hilbert space $\cH.$ Then there exists a unique $\lambda\in\Delta_q,$ and unique orthonormal base $\set{e_k}_{k\in\dZ}$ in $\cH$ such that
\begin{gather}\label{eq_irr_qnorm}
X e_k=\lambda q^{-k/2}e_{k+1},\ X^*e_k=\lambda q^{-(k-1)/2}e_{k-1},\ k\in\dZ.
\end{gather}
\end{prop}
\begin{proof}
Since $X,\ X\neq 0$ is irreducible, by Theorem \ref{thm_spec} we get $X=X_\mu$ for some Borel measure $\mu$ on $(0,+\infty)$ satisfying (\ref{eq_muD}). It follows from the proof of Theorem \ref{thm_spec} that operator $D=C_\mu E_{C_\mu}(\Delta_q)$ is irreducible, i.e. one-dimensional. In particular $\supp\mu\cap\Delta_q$ consists of a singular point $\lambda\in\Delta_q.$ Put
$$
e_k=\One_{\set{\lambda q^{-k/2}}}(\mu(\set{\lambda q^{-k/2}}))^{-1},\ k\in\dZ.
$$

Direct computations show that (\ref{eq_irr_qnorm}) is satisfied. Further, one can check that every bounded operator which commutes with $X$ and $X^*$ is a multiple of identity. Hence $X$ is irreducible.
\end{proof}

Below we will often use the following
\begin{lemma}\label{lemma_dense}
Let $q>0,$ $X$ be a $q$-normal operator on a Hilbert space $\cH$ and let $X=UC$ be its polar decomposition.
\begin{enumerate}
  \item[$(i)$] For all $m,n\in\dN_0$
  \begin{gather}\label{eq_XnXm}
X^{*m}X^n=q^{(m^2+m-n^2+n-2mn)/4}U^{n-m}C^{m+n},
\end{gather}
Where $U^{-k}$ denotes $U^{*k},\ k\in\dN.$ In particular, $\cD^\infty(X)=\cap_{m,n\in\dN}\cD(X^{*m}X^n)=\cD^\infty(C)$ is dense in $\cH.$
  \item[$(ii)$] The set $\cD^\infty(X)$ is a core of $X^{*m}X^n$ for all $m,n\in\dN_0.$
  \item[$(iii)$] The set $\cD^\infty(X)$ is invariant under $U$ and $U^*.$
\end{enumerate}
\end{lemma}
\begin{proof}
\noindent$(i):$ By Proposition \ref{prop_char}, $(iv)$ we have $UC=q^{1/2}CU,$ which implies
\begin{gather*}
X^{*m}X^n=(CU^*)^{m}(UC)^n=q^{[(m(m+1))/4-(n(n-1))/4]}U^{*m}C^mU^nC^n=\\
=q^{(m^2+m-n^2+n-2mn)/4}U^{n-m}C^{m+n}.
\end{gather*}

\noindent$(ii):$ Since $\cD^\infty(X)=\cD^\infty(C)$ and $C$ is self-adjoint, $\cD^\infty(X)$ is a core of $C^m,\ m\in \dN.$ It follows from (\ref{eq_XnXm}) that $\norm{X^{*m}X^n\varphi}=q^{(m^2+m-n^2+n-2mn)/4}\norm{C^{m+n}\varphi},\ \varphi\in\cD^{\infty}(C),$ which implies the assertion.

\noindent$(iii):$ By Proposition \ref{prop_char}, $(iv),$ we have $UC^k=q^{k/2}C^kU,\ k\in\dZ.$ Hence $\cD(C^k)$ is invariant for $U$ for all $k\in\dZ,$ i.e. $\dC^\infty$ is invariant for $U.$ In the same way one shows that $\cD^\infty(X)$ is invariant for $U^*.$
\end{proof}



\section{Positive $q$-polynomials}\label{positive}
Recall that
$$
\cA=\dC\langle x,x^*|\ xx^*=qx^*x\rangle,
$$
where $q$ is a positive real number. Since the set $\{x^{*m}x^n; m,n\in \dN_0\}$ is a vector space basis of $\cA$, each
element $f\in\cA$ can be written uniquely as
 $f=\sum_{m,n}\alpha_{mn}x^{*m}x^n,$ where $a_{mn}\in \dC$.
We define the \emph{degree} of $f$ by $\deg f:=\max\set{m{+}n|\ \alpha_{mn}\neq 0}.$ We will also refer to an element $f$ of $\cA$ as a \emph{$q$-polynomial} and write $f=f(x,x^*).$ If $X$ is a $q$-normal operator, then $f(X,X^*)$ denotes the operator $\sum_{m,n}\alpha_{mn}X^{*m}X^n.$

\begin{defn}\label{defn_pos}
An element $f=f^*\in\cA$ is called \emph{positive} if
\begin{align*}
\langle f(X,X^*)\varphi,\varphi\rangle\geq 0\ \mbox{for}\ \varphi\in\cD^\infty(X)
\end{align*}
for every $q$-normal operator $X$ and every $\varphi\in\cD^\infty(X).$ The set of positive elements of $\cA$ is denoted by $\cA_+.$
\end{defn}
 With this notion of positivity one can develop a non-commutative real algebraic geometry on the complex $q$-plane. In this section we investigate positive elements and sum of squares in $\cA$. In Section \ref{strictpositivstellensatz} we prove a strict Positivstellensatz for $\cA.$

\begin{defn}\label{defn_wellbeh}
A $*$-representation $\pi$ of $\cA$ is called {\it well-behaved} if there is a $q$-normal operator $X$ such that $\cD^\infty(X)=\cD(\pi)$ and $\pi(x)=X\upharpoonright\cD^\infty(X).$
\end{defn}
By Lemma \ref{lemma_dense} the domain $\cD^\infty(X)$ is a core of a $q$-normal operator $X,$ so there is a one-to-one correspondence between $q$-normal operators and well-behaved representations of $\cA$.

Further, an element $f=f^\ast\in\cA$ is in $\cA_+$ if and only if $\pi(f)\geq 0$ for every well-behaved $*$-representation $\pi.$ Thus our definition of positive elements fits into the definition of positivity via $\ast$-representations proposed in \cite{s2}.

\medskip
\noindent\textbf{Remarks.} 1. If $\mu$ is a positive Borel measure on $\dR_+$ satisfying (\ref{eq_muD}) and $X_\mu$ is the $q$-normal operators defined by (\ref{eq_Amu}), we denote the corresponding well-behaved $*$-representation of $\cA$ by $\pi_\mu.$ That is,
\begin{gather}\label{eq_pimu}
\pi_\mu(x)=X_\mu\upharpoonright\cD^\infty(X_\mu),\ \cD^\infty(X_\mu)=\cD(\pi_\mu).
\end{gather}

\noindent 2. The $\ast$-algebra $\cA$ has a natural $\dZ$-grading given by $\deg x=1$ and $\deg x^*=-1.$ In \cite{ss} a notion of well-behaved $*$-representations was introduced for class of group graded $*$-algebras which contains $\cA.$ It can be shown that Definition \ref{defn_wellbeh} is equivalent to corresponding definition of well-behavedness in \cite{ss}, see Definition 11 therein.

\medskip
Suppose that $p\in\dR[t].$ We consider the following questions:

\begin{center}{\it When $p(x+x^*)\in\cA_+$? When $p(x+x^*)\in\sum\cA^2$?}\end{center}
First we consider the case when $\deg p=2.$

Below we will need the following
\begin{lemma}\label{lemma_SDP}
Let $f=\sum_{m,n}\alpha_{mn}X^{*m}X^n\in\cA,\ \deg f=2N,\ N\in\dN.$ Denote by $w_N$ the column vector of monomials
$$
w_N=(1\quad x\quad x^*\quad x^2\quad x^*x\quad x^{*2}\ \dots\ x^{*N})^T,
$$
and let $w_N^*=(1\quad x\quad x^*\quad x^{*2}\quad x^*x\quad x^2\dots).$
Then $f$ is a sum of squares in $\cA$ if and only if there exists a positive semidefinite $(N+1)(N+2)/2\times(N+1)(N+2)/2$ complex matrix $C$ such that $f=w_N^*Cw_N^{}.$
\end{lemma}
\begin{proof}
Asume $f=\sum_i f_i^*f_i^{},\ f\in\cA.$ Then $\deg f_i\leq N,$ and there exist row vectors $a_i\in\dC^{(N+1)(N+2)/2}$ such that $f_i=a_i w_N.$ It implies that $f=\sum_i (a_i w_N)^*a_i w_N=w_N^*Cw_N^{},$ where $C=\sum_i a_i^*a_i$ is positive semidefinite.

On the other hand, if $f=w_N^*Cw_N^{},\ C\geq 0,$ then $C$ is a sum of rank one positive semidefinite matrices. That is there exist row vectors $a_i\in\dC^{(N+1)(N+2)/2}$ such that $C=\sum_i a_i^*a_i,$ which implies that $f=\sum_i (a_i w_N)^*a_i w_N\in\sum\cA^2.$
\end{proof}

\begin{prop}\label{prop_sos_deg2}
Let $a, b\in \dR$.
The element $L:=(x+x^*)^2-2a(x+x^*)+b$ is in $\sum \cA^2$ if and only if
$b \geq \dfrac{4a^2 q}{(q+1)^2}.$
\end{prop}
\begin{proof} First suppose that $L\in \sum \cA^2$. Then by Lemma \ref{lemma_SDP} there is a positive semi-definite matrix $C=[c_{i,j}]_{i,j=\ov{1,3}}$ such that $L=w_1^*Cw_1^{}$.
Comparing coefficients at $x^{*m}x^n,\ m+n\leq 2$ yields
\begin{eqnarray*}
b & = & c_{11} \\
-2a &= & c_{13}+c_{21} = c_{31}+c_{12} \\
1 & = & c_{23}=c_{32} \\
q+1 & = & c_{22}+qc_{33}
\end{eqnarray*}
By multiplying these equations with $\alpha_1=1$, $\alpha_2=\frac{2aq}{(1+q)^2}$,
$\alpha_3=0$ and $\alpha_4=\frac{4a^2 q^2}{(1+q)^3}$, respectively, and adding them we derive
$$
b-\frac{4a^2 q}{(q+1)^2}=\tr
\left[ \begin{array}{ccc} \alpha_1 & \alpha_2 & \alpha_2 \\ \alpha_2 & \alpha_4 & \alpha_3 \\\alpha_2 & \alpha_3 & q\alpha_4 \end{array} \right]
\left[ \begin{array}{ccc} c_{11} & c_{12} & c_{13} \\ c_{21} & c_{22} & c_{23} \\c_{31} & c_{32} & c_{33} \end{array} \right].
$$
Both matrices in the preceding equation are positive semidefinite. Therefore, since the trace of the product of two positive semidefinite matrices is nonnegative, $b-\frac{4a^2 q}{(q+1)^2}\geq 0$.

Conversely, suppose that $b \geq 4a^2 q(q+1)^{-2}.$ Setting
$$
f=q^{-1/2}(-2aq(1+q)^{-1}+qx+x^*),
$$
we compute
\begin{gather}\label{eq_c4}
(x+x^*)^2-2a(x+x^*)+\frac{4a^2 q}{(q+1)^2}= f^*f.
\end{gather}
Hence $L=f^*f+((b {-} 4a^2 q(q+1)^{-2})^{1/2}\Un)^2.$
\end{proof}

The preceding result has the following interesting application.
\begin{prop}
Let $X\neq 0$ be a $q$-normal operator. Then each non-zero real number is a regular point for the symmetric operator $X+X^*$. In particular, $X+X^*$ is not essentially self-adjoint.
\end{prop}
\begin{proof}
Suppose that $a\in\dR\setminus\set{0}$. It follows from (\ref{eq_c4}) that the element
$$
(x+x^*)^2-2a(x+x^*)+\frac{4a^2 q}{(q+1)^2}=((x+x^*)-a)^2-a^2\left(\frac{q-1}{q+1}\right)^2
$$
is a square in $A.$ This implies that
\begin{align}\label{regpoint}
\norm{(X+X^*-a)\varphi}\geq a\left|\frac{q-1}{q+1}\right|\norm{\varphi},
\end{align}
for $\varphi\in\cD^\infty(X)$. Since $\cD^\infty(X)$ is a core for $X+X^\ast$ by Lemma \ref{lemma_dense} $(ii)$, the inequality (\ref{regpoint}) holds for all $\varphi\in \cD(X{+}X^\ast)$. This shows that $a$ is a regular point for $X+X^*.$

Let $T$ denote the closure of $X+X^*$ and assume to the contrary that $T$ is self-adjoint. The numbers of $\dR\setminus\set{0}$ are regular points for $X+X^\ast$ and hence for the selfadjoint operator $T$. Therefore, $\dR\setminus\set{0}\subseteq \rho(T),$ so that $\sigma(T)=\set{0}$. The latter implies that $X+X^*=0$ which is impossible for $X\neq 0.$
\end{proof}

For a positive Borel measure $\mu$ on $\dR_+$ satisfying (\ref{eq_muD}) we define a linear functional $F_\mu$ on $\cA$ by
\begin{gather}\label{eq_pmu}
F_{\mu}(f)=\langle\pi_{\mu}(f)\One_{\Delta_q},\One_{\Delta_q}\rangle,~~f\in \cA,
\end{gather}
where $\pi_\mu$ is defined by (\ref{eq_pimu}). 
\begin{lemma}\label{lemma_cyc}
Let $\pi_\mu$ be the $*$-representation of $\cA$ defined by (\ref{eq_pimu}) and let $f=f^*\in\cA.$ Then we have $\pi_\mu(f)\geq 0$ if and only if
\begin{gather}\label{eq_pos_pmu}
F_\mu(g^*f g)\geq 0\ \mbox{for all}\ g\in\cA.
\end{gather}
\end{lemma}
\begin{proof}
Let $C_\mu$ be as in (\ref{eq_Umu_Cmu}). It follows from relation (\ref{eq_XnXm}) that the graph topology on $\cD(\pi_\mu)=\cD^\infty(C_\mu)$ is generated by the family of seminorms $\norm{C_\mu^n(\cdot)},\ n\in\dN_0.$ Set $\varphi_0=\One_{\Delta_q}.$ For $k\in\dZ$ define
$$
\varphi_k:=\left\{\begin{array}{ll}
            (X_\mu^k\varphi_0)(t), & \hbox{if $k\geq 0$;} \\
            (X_\mu^{*|k|}\varphi_0)(t), & \hbox{if $k<0$.}
            \end{array}
           \right.
$$
Then $\varphi_k=q^{k/2}t^k\One_{\set{q^{-k/2}\Delta_q}}(t)$ if $k\geq 0$ and $\varphi_k=t^{|k|}\One_{\set{q^{|k|/2}\Delta_q}}(t)$ if $k<0$. Since $X_\mu^*X_\mu^{}=C_\mu^2,$ for each $k\in \dZ$ the set $
\set{p(X_\mu^*X_\mu^{})\varphi_k|\ p\in\dC[t]}$
is dense in the subspace $L_2(q^{-k/2}\Delta_q,d\mu)\subseteq\cD(\pi_\mu)$ with respect to the graph topology.

Consider the case $\mu(\set{0})=0.$ Then $\cD(\pi_\mu)$ is a direct sum of $L_2(q^{k/2}\Delta_q,d\mu),\ k\in\dZ.$ Hence $\varphi_0$ is cyclic for $\pi_\mu.$ Therefore, since $F_\mu(g^*fg)=\langle \pi_\mu(f)\pi_\mu(g)\varphi_0,\pi_\mu(g)\varphi_0\rangle$ and $\pi_\mu(\cA)\varphi_0$ is dense in $\cD(\pi_\mu)$ in the graph topology, it follows that $\pi_\mu(f)\geq 0$ if and only if condition (\ref{eq_pos_pmu}) is satisfied.


If $\supp\mu=\set{0},$ then the statement is trivial. Consider the case $\supp\mu\neq\set{0},\ \mu(\set{0})\neq 0$ and let $\mu_1$ be the Borel measure on $\dR_+$ defined by $\mu_1(\Delta)=\mu(\Delta\setminus\set{0}).$ Then $\pi_\mu(x)$ is a direct sum of $0$ and $\pi_{\mu_1}(x).$ Let $f=f^*\in\cA(q)$ satisfy (\ref{eq_pos_pmu}). Since $F_\mu=F_{\mu_1},\ f$ is positive in $\pi_1.$ It suffices to prove $f(0,0)\geq 0.$ For let $f=\sum_{m,n}\alpha_{mn}x^{*m}x^n$ and $q>1.$ Then for $k\in\dN$ we have
\begin{gather*}
q^{-k}F_\mu(x^{*k}fx^k)
=q^{-k}\sum_{m,n}\alpha_{mn}\langle \pi_\mu(x^{*(k+m)}x^{k+n})\varphi_0,\varphi_0\rangle=\\
=q^{-k}\sum_{m,n}\alpha_{mn}\langle X_\mu^{k+n}\One_{\Delta_q},X_\mu^{k+m}\One_{\Delta_q}\rangle=\sum_{n}q^n\alpha_{nn}\int t^{2n}\One_{q^{-(k+n)/2}\Delta_q} dt
\end{gather*}
which converges to $\alpha_{00}=f(0,0),$ for $k\to\infty.$ In the case $q<1$ we have $q^{k}F_\mu(x^kfx^{*k})\to f(0,0),\ k\to\infty.$ It implies $f(0,0)\geq 0.$
\end{proof}
\medskip
\noindent\textbf{Remark.} In the case $\mu(\{0\})=0$ we have seen in the preceding proof that the vector $\varphi_0=\One_{\Delta_q}$ is cyclic for $\pi_\mu.$
Therefore, by uniqueness of GNS-re\-pre\-sen\-ta\-tion (see Theorem 8.6.4. in \cite{s4}), $\pi_{\mu}$ is unitarily equivalent to the GNS-re\-pre\-sen\-ta\-tion of the positive functional $F_\mu$.

Denote by $\cB=\dC[x^*x]$ the unital $*$-subalgebra of $\cA$ generated by the single element $x^*x.$ For every element $g\in\cB$ there exists a unique polynomial, denoted by $g(t)\in\dC[t],$ such that $g=g(x^*x).$ For $f=\sum_{m,n}\alpha_{mn}x^{*m}x^n\in\cA(q)$ we define
$$
\gp(f):=\sum_{n}\alpha_{nn}x^{*n}x^n=\sum_{n}\alpha_{nn}q^{n(n-1)/2}(x^*x)^n\in\cB.
$$

Then the mapping $\gp:\cA\to \cB$ is a conditional expectation as introduced in \cite{ss}. We collect some properties of $\gp$ in a lemma. We omit its simple proof.

\begin{lemma}\label{lemma_prop_gp}
Let $q>0.$
\begin{enumerate}
  \item[$(i)$] For every positive functional $F_\mu$ defined by (\ref{eq_pmu}) and every $f\in\cA$ holds $F_\mu(f)=F_\mu(\gp(f)).$ In particular,
      \begin{gather*}
        F_\mu(f)=\int_{\Delta_q}(\gp(f))(t)d\mu(t^{1/2}).
      \end{gather*}

  \item[$(ii)$] For $f\in\cA$ and $g_1,g_2\in\cB$, we have $\gp(g_1^*fg_2^{})=g_1^*\gp(f)g_2^{}.$
\end{enumerate}
\end{lemma}

\begin{prop}\label{prop_stand_qnorm}
Let $\cH=L_2(\dR,d\lambda),$ where $d\lambda$ is the Lebesgue measure on $\dR$. We define operators $U_0$ and $C_0$ on $\cH$ by
\begin{gather}\label{eq_qnorm_st}
(U_0\varphi)(t)=\varphi(t+1),\ (C_0\psi)(t)=q^{t/2}\psi(t),
\end{gather}
where $\cD(U_0)=\cH,\ \cD(C_0)=\set{\psi|\ t\psi(t)\in L_2(\dR,d\lambda)}$. Then $X_0:=U_0C_0$ is a $q$-normal operator and an element $f=f^*\in\cA$ is in $\cA_+$ if and only if
\begin{gather}\label{eq_pos_stand}
\langle f(X_0^{},X_0^*)\psi,\psi\rangle\geq 0\ \mbox{for all}\ \psi\in\cD^\infty(X_0).
\end{gather}
\end{prop}
\begin{proof}
Let $\mu_0$ be a measure on $\dR_+$ satisfying (\ref{eq_muD}) such that $\mu_0\upharpoonright\Delta_q$ coincides with the Lebesgue measure $d\lambda.$ Then the unitary operator $\cU: L_2(\dR,d\lambda)\to L_2(\Delta_q,d\mu_0),\ (\cU\psi)(t)=\psi(q^{t/2})$ defines a unitary equivalence of $X_0$ and $X_{\mu_0}.$ Hence $X_0$ is $q$-normal and for each $f\in\cA_+$ condition (\ref{eq_pos_stand}) is satisfied.

Conversely, suppose that (\ref{eq_pos_stand}) holds.
Let $\pi_{\mu_0}$ and $F_{\mu_0}$ be the $*$-representation and positive functional defined by (\ref{eq_pimu}) and (\ref{eq_pmu}) respectively. Since (\ref{eq_pos_stand}) holds, we have $F_{\mu_0}(g^*fg)\geq 0$ for all $g\in\cA$ and hence
\begin{align*}
F_{\mu_0}(h^*g^*f g h)\geq 0~~{\rm for ~all}~~ h\in\cB, g\in\cA(q).
\end{align*}
Using Lemma \ref{lemma_prop_gp} we obtain for a fixed $g\in\cA$ and every $h\in\cB$,
\begin{gather*}
F_{\mu_0}(h^*g^*f g h)=p_{\mu_0}(\gp(h^*g^*f g h))=\int_{\Delta_q}h^*(t)(\gp(g^*f g ))(t)h(t) d\lambda(t^{1/2})=\\
=\int_{\Delta_q}(\gp(g^*f g ))(t)|h(t)|^2 d\lambda(t^{1/2})\geq 0.
\end{gather*}
Since the polynomials $h\in \dC[t]$ are dense in $L^2(\Delta_q,d\lambda)$ it follows that
\begin{align*}
(\gp(g^*f g ))(t)\geq 0~~ {\rm for}~~\ t\in\Delta_q, ~ g\in\cA.
\end{align*}

Let $\mu$ be another measure on $\dR_+$ satisfying (\ref{eq_muD}) and $F_\mu$ be the corresponding positive functional. We assume first that $\mu(\set{0})=0.$ Then
\begin{align*}
F_{\mu}(g^*f g)=\int_{\Delta_q}(\gp(g^*f g ))(t)d\lambda(t^{1/2})\geq 0~~ {\rm for~ all}~~ g\in\cA
\end{align*}
which implies that $\pi_\mu(f)\geq 0$.
\end{proof}

\begin{thm}\label{thm_motzkin1}For $c\in \dR$, set $L_c:=(x+x^*)^4-2(x+x^*)^2+c.$ Then:
\begin{enumerate}
\item[$(i)$] $L_c\in \sum \cA^2$ if and only if $c\geq\frac{q (q+1)^2}{\left(q^2+1\right)^2}$,
\item[$(ii)$] there exists $\varepsilon>0$ such that $L:=(x+x^*)^4-2(x+x^*)^2+\frac{q (q+1)^2}{\left(q^2+1\right)^2}-\varepsilon\in\cA_+$.
\end{enumerate}
\end{thm}

\begin{proof} (i): First suppose that $L_c\in \sum \cA^2$.
By Lemma \ref{lemma_SDP} there exists a positive semidefinite $6\times 6$-matrix $C$ such that $L=w_2^\ast Cw_2^{}$. Comparing coefficients at $x^{*m}x^n$ in the equation $L=w_2^\ast Cw_2^{}$, we obtain the following equations

$$
\begin{array}{rcl}
c & = & c_{1,1},   \\
0 & = & c_{1,2}+c_{3,1}=c_{2,1}+c_{1,3},   \\
-2 & = & c_{1,4}+c_{3,2}+c_{6,1}= c_{4,1}+c_{2,3}+c_{1,6},  \\
-2q-2 & = & c_{1,5}+ c_{2,2}+q c_{3,3}+c_{5,1},  \\
0 & = & c_{2,6}+c_{4,3}= c_{6,2}+c_{3,4},  \\
0 & = & q c_{3,5}+ c_{2,4}+c_{5,2}+q^2 c_{6,3}= q c_{5,3}+ c_{4,2}+c_{2,5}+q^2 c_{3,6},  \\
1 & = & c_{4,6} = c_{6,4},  \\
(1+q)(1+q^2) & = & q^2 c_{5,6}+c_{4,5} = q^2 c_{6,5}+ c_{5,4},   \\
(1+q^2)(1+q+q^2) & = &  c_{4,4}+q c_{5,5}+q^4 c_{6,6},   \\
\end{array}
$$

Multiplying each line with the respective $\alpha_i$,
$$\begin{array}{rcl}
\alpha_1 & = & 1,\\
\alpha_2 & = & 0,\\
\alpha_3 & = & 0,\\
\alpha_4 & = & \frac{q + q^2}{(1 + q^2)^2},\\
\alpha_5 & = & 0,\\
\alpha_6 & = & 0,\\
\alpha_7 & = & \frac{q^3 (1 + q)^2}{(-1 + q)^2 (1 + q^2)^2 (1 + q + q^2)},\\
\alpha_8 & = & -\frac{q^2 (1 + q)^3}{(-1 + q)^2 (1 + q^2)^3 (1 + q + q^2)},\\
\alpha_9 & = & \frac{q (1 + q)^2}{(-1 + q)^2 (1 + q^2)^2 (1 + q + q^2)},\\
\end{array}$$
and adding them we get
$$
c - \frac{q (1 + q)^2}{(1 + q^2)^2}=\mathrm{Tr} \left(
\left[\begin{array}{cccccc}
\alpha_1 & \alpha_2 & \alpha_2 & \alpha_3 & \alpha_4 & \alpha_3 \\
 \alpha_2 & \alpha_4 & \alpha_3 & \alpha_6 & \alpha_6 & \alpha_5 \\
 \alpha_2 & \alpha_3 & q \alpha_4 & \alpha_5 & q \alpha_6 & q^2 \alpha_6 \\
 \alpha_3 & \alpha_6 & \alpha_5 & \alpha_9 & \alpha_8 & \alpha_7 \\
 \alpha_4 & \alpha_6 & q \alpha_6 & \alpha_8 & q \alpha_9 & q^2 \alpha_8 \\
 \alpha_3 & \alpha_5 & q^2 \alpha_6 & \alpha_7 & q^2 \alpha_8 & q^4\alpha_9
\end{array}\right]C\right)
$$
By some simple computations one checks that the matrix containing the $\alpha_i$ is positive semi\-de\-fi\-nite. Since $C$ is also positive semidefinite, it follows from the preceding that
\begin{align}\label{inequl}
c \geq \frac{q (1 + q)^2}{(1 + q^2)^2}.
\end{align}

Conversely, suppose that (\ref{inequl}) is satisfied. Setting
$$u_1(x,x^*) = -\frac{q}{1 + q^2} + \frac{x^{*2}}{q (1 + q)} + \frac{(1 + q^2) x^*x}{1 + q} + \frac{q^2 x^2}{1 + q}$$
and
$$u_2(x,x^*) = -\frac{q}{1 + q^2} + \frac{x^{*2}}{1 + q} + \frac{(1 + q^2) x^*x}{1 + q} + \frac{q x^2}{1 + q},$$
we compute
\begin{gather}\label{eq_L_sos}
(x+x^*)^4-2(x+x^*)^2+\frac{q (1 + q)^2}{(1 + q^2)^2}= u_1^\ast u_1^{} + \frac{1 + q + q^2}{q} u_2^\ast u_2^{}.
\end{gather}
This implies that $L\in \sum \cA^2$ if (\ref{inequl}) holds.

\medskip
\noindent$(ii)$: Since the element $L$ remains invariant if we replace $x$ by $x^*$ and $q$ by $q^{-1}, $ it suffices to treat the case $q>1.$ Assume to the contrary that no such $\varepsilon>0$ exists. Let $X_0$ be the $q$-normal operator from Proposition \ref{prop_stand_qnorm}. Then there exists a sequence of unit vectors $\varphi_n\in\cD^{\infty}(X_0)$ such that $\langle L(X_0^{},X_0^*)\varphi_n,\varphi_n\rangle\to 0$ as $n\to\infty.$ It follows from (\ref{eq_L_sos}) that
\begin{align}
u_1(X_0^{},X_0^*)\varphi_n\to 0\ \mbox{and}\ u_2(X_0^{},X_0^*)\varphi_n\to 0\ \mbox{as}\ n\to\infty.
\end{align}

Put $X:=X_0^2.$ Then $X$ is a $q^4$-normal operator, $X^*=X_0^{*2}$ and $X_0^*X_0=(qX_0^{*2}X_0^2)^{1/2}=q^{1/2}|X|.$ Define operators
\begin{gather*}
F_1=\frac{q(q+1)}{(q-1)}(u_1(X_0^{},X_0^*)-u_2(X_0^{},X_0^*))=X_0^{*2}-q^2X_0^2=X^*-q^2X,\\ 
\nonumber F_2=\frac{1}{q-1}(q u_1(X_0^{},X_0^*)-u_2(X_0^{},X_0^*))=qX+\frac{1+q^2}{1+q}q^{1/2}|X|-\frac{q}{1+q^2}.
\end{gather*}
Then $F_1\varphi_n\to 0$ and $F_2\varphi_n\to 0$. Let $X=UC$ be the polar decomposition of $X.$ From Proposition \ref{prop_char} we get $X^*-q^2X=q^2U^*(I-U^2)C$ which implies that
\begin{align}\label{eq_c1}
(I-U^2)C\varphi_n\to 0.
\end{align}
Put $\alpha=q,\ \beta=q^{1/2}\frac{1+q^2}{1+q},\ \gamma=-\frac{q}{1+q^2},$ so that $F_2=\alpha UC+\beta C+\gamma.$ Then $(I-U^2)F_2\varphi_n\to 0.$ Combined with (\ref{eq_c1}) it follows that
\begin{align}\label{eq_c2}
(I-U^2)\varphi_n\to 0.
\end{align}
Let $\dT_+=\set{e^{\img t},\ t\in[-\pi/2,\pi/2)},\ \dT_-=\set{e^{\img t},\ t\in[\pi/2,3\pi/2)}.$ We put $\xi_n=E_U(\dT_+)\varphi_n,\ \psi_n=E_U(\dT_-)\varphi_n.$ Then $\varphi_n=\psi_n+\xi_n$ and (\ref{eq_c2}) yields
\begin{gather}\label{eq_c3}
(U-I)\xi_n\to 0~~{\rm and}~~ (U+I)\psi_n\to 0.
\end{gather}
Since $C$ is positive, $C+1$ is invertible, so that $(C+1)^{-1}F_2\varphi_n\to 0.$ By $UC=q^2CU,$
$$
(C+1)^{-1}(\alpha q^2CU+\beta C+\gamma)\varphi_n\to 0.
$$
Using (\ref{eq_c3}) we obtain
$$
\frac{\alpha q^2C}{C+1}(\xi_n-\psi_n)+\frac{\beta C+\gamma}{C+1}(\xi_n+\psi_n)\to 0
$$
which implies that
$$
\frac{(\alpha q^2+\beta)C+\gamma}{C+1}\xi_n-\frac{(\alpha q^2-\beta)C-\gamma}{C+1}\psi_n\to 0.
$$
Since $q>1$, $\alpha q^2-\beta>0$ and $\gamma<0$. Hence the operator $(\alpha q^2-\beta)C-\gamma$ has a bounded inverse. Applying this inverse to the preceding equation we get
\begin{gather}\label{eq_c5}
-\psi_n+\frac{(\alpha q^2+\beta)C+\gamma}{(\alpha q^2-\beta)C-\gamma}\xi_n\to 0.
\end{gather}
Applying $U$ and using Proposition \ref{prop_char} and (\ref{eq_c3}) we derive
\begin{gather}\label{eq_c6}
\psi_n+\frac{(\alpha q^2+\beta)q^2C+\gamma}{(\alpha q^2-\beta)q^2C-\gamma}\xi_n\to 0.
\end{gather}
Adding (\ref{eq_c5}) and (\ref{eq_c6}) we obtain
\begin{gather}\label{eq_c7}
\frac{\alpha_1 C^2+\beta_1 C+\gamma_1}{((\alpha q^2-\beta)q^2C-\gamma)((\alpha q^2-\beta)C-\gamma)}\xi_n\to 0,\ n\to\infty,
\end{gather}
where $\alpha_1=2q^2(\alpha^2q^4-\beta^2),\ \beta_1=-2\beta\gamma(1+q^2),\ \gamma_1=-2\gamma^2.$ The polynomial $\alpha_1 c^2+\beta_1 c+\gamma_1$ has two real roots $c_1<0<c_2.$

Since $X=(U_0C_0)^2=q^{-1/2}U_0^2C_0^2,$ we get $U=U_0^2,\ C=q^{-1/2}C_0^2.$ By (\ref{eq_qnorm_st}),
\begin{align*}
(U\varphi)(t)=\varphi(t+2),~(C\psi)(t)=q^{t-1/2}\psi(t)~~{\rm for}~~
\varphi\in L_2(\dR,d\lambda), \psi\in\cD(C_0^2).
\end{align*}
Set $t_2=\log_{q}c_2+1/2.$ Then (\ref{eq_c7}) implies that
$$\norm{\xi_n}^2-\int_{t_2-1/2}^{t_2+1/2}|\xi_n|^2d t\to 0.$$
Hence $\langle U\xi_n,\xi_n\rangle=\int_{\dR}\xi_n(t+2)\overline{\xi_n(t)}d t\to 0.$ Combined with (\ref{eq_c3}) this yields $\xi_n\to 0.$ Further, (\ref{eq_c5}) implies $\psi_n\to 0$ which contradicts $\norm{\xi_n+\psi_n}=\norm{\varphi_n}=1.$
\end{proof}

We end up the section with the following
\begin{prop}
Let $f\in\dR[t]$. If $f(x^*x)\in\cA(q)_+,$ then $f(x^*x)\in\sum\cA(q)^2.$
\end{prop}
\begin{proof}
Let us choose a measure $\mu$ satisfying (\ref{eq_muD}) such that $\supp~\mu=\dR_+$ and let $X_\mu$ be the operator defined by (\ref{eq_Amu}). Then the spectrum of $X_\mu^*X_\mu^{}$ is equal to $\dR_+.$ Therefore, since $f(X_\mu^*X_\mu^{})\geq 0,$ we have $f\geq 0$ on $\dR_+.$ Hence (see e.g. \cite{m}) there exist polynomials $g_1,g_2\in\dC[t]$ such that $f(t)=g_1(t)^*g_1(t)+t\cdot g_2(t)^*g_2(t).$ Then
\begin{gather*}
f(x^*x)=g_1(x^*x)^*g_1(x^*x)+x^*x\cdot g_2(x^*x)^*g_2(x^*x)=\\
=g_1(x^*x)^*g_1(x^*x)+x^*\cdot g_2(qx^*x)^*g_2(qx^*x)\cdot x\in\sum\cA(q)^2.
\end{gather*}
\end{proof}

\section{The complex $q$-moments problem and formally $q$-normal operators}\label{q-momentproblem}
\begin{defn}
A linear functional $F$ on $\cA$ is called a \textit{$q$-moment functional} if there exists a well-behaved $\ast$-representation $\pi$ of $\cA$ and a vector $\xi \in \cD(\pi)$ such that
\begin{gather}
F(a)=\langle \pi(a)\xi,\xi\rangle ~~{\rm for~all}~~ a\in \cA.
\end{gather}
\end{defn}
\noindent Then the \textit{$q$-moment problem} asks:
\begin{center}
 \textit{When is a given functional $F$ on $\cA$ a $q$-moment functional?}
 \end{center}
In this formulation the $q$-moment problem is a \textit{generalized moment problem} in the sense of \cite{s5}. Next we give two reformulations of the $q$-moment problem.

Since $\{x^{*k}x^l;k,l\in \dN_0\}$ is a vector space basis of $\cA$, there is a one-to-one-correspondence between complex $2$-sequences and linear functionals on $\cA$ given by $F_a(x^{*k}x^l)=a_{kl}$, $k,l\in \dN_0$, where $a=(a_{kl})_{k,l\in\dN_0}$ is a $2$-sequence.
The definition of a well-hehaved representation (Definition \ref{defn_wellbeh})
yields the following equivalent formulation of the $q$-moment problem:
\begin{center}
\textit{Given a $2$-sequence $(a_{kl})_{k,l\in\dN_0},$ does there exist a $q$-normal operator $X$ and a vector $\xi\in\cD^{\infty}(X)$ such that}
\begin{gather}\label{eq_mom_seq}
a_{kl}=\langle X^{*k}X^l\xi,\xi\rangle\ \mbox{for all}\ \ k,l\in\dN_0?
\end{gather}
\end{center}
Before we turn to the second reformulation we consider an example.

\medskip
\noindent\textbf{Example.} Suppose that $\mu$ is a positive Borel measure on $\Delta_q$. Denote by $\muw$ the unique extension of $\mu$ to a measure on $\dR_+$ satisfying (\ref{eq_muD}). Let $X_{\mu}$ be the $q$-normal operator defined by (\ref{eq_Amu}) and $\xi\in\cD^\infty(X_\mu)$. Then there is a $q$-moment functional defined by $F_{\mu,\xi}(f(x,x^*)):=\langle f(X_\mu,X_\mu^*)\xi,\xi\rangle.$ Since
\begin{gather*}
(X_\mu^{*k}X_\mu^l\xi)(t)=q^{(l^2+l-k^2+k-2kl)/4}t^{k+l}\xi(q^{(l-k)/2}t),\ k,l\in\dN_0,
\end{gather*}
by Lemma \ref{lemma_dense}, the corresponding $q$-moments are
\begin{gather}
a_{kl}=F_{\mu,\xi}(x^{*k}x^l)=\langle X_\mu^{*k}X_\mu^l\xi,\xi\rangle=\int_{\dR_+}q^{(l^2+l-k^2+k-2kl)/4}t^{k+l}\xi(q^{(l-k)/2}t)\overline{\xi(t)}d\mu(t)\nonumber\\
\label{muqmoments}=q^{(l^2+l-k^2+k-2kl)/4}\int_{\dR_+}(q^{k/2}t)^{k+l}\xi(q^{l/2}t)\overline{\xi(q^{k/2}t)}d\tilde{\mu}(q^{k/2}t)=\\
\nonumber=q^{(l^2+l+k^2+k)/4}\int_{\dR_+}t^{k+l}\xi(q^{l/2}t)\overline{\xi(q^{k/2}t)}d\tilde{\mu}(t).
\end{gather}

Using Theorem \ref{thm_spec} and formula (\ref{muqmoments}) we obtain another equivalent formulation of the $q$-moment problem in terms of measures and integrals:
\begin{center}
\textit{Given a $2$-sequence $(a_{kl})_{k,l\in\dN_0},$ does there exist a family $\mu_i,i\in I$, of positive Borel measures on $\Delta_q$ and a vector $\xi=(\xi_i)\in \bigoplus_i\cD^\infty(X_{\mu_i})$ in the Hilbert space $\bigoplus_i L^2(\dR_+,\muw_i)$ such that}
\begin{gather*}
a_{kl}=\sum_i q^{(l^2+l+k^2+k)/4}\int_{\dR_+}t^{k+l}\xi_i(q^{l/2}t)\overline{\xi_i(q^{k/2}t)}d\muw_i(t)\ \mbox{for}\ \ k,l\in\dN_0?
\end{gather*}
\end{center}

The next theorem is the counter-part of Haviland's theorem from the classical moment problem. For this we need the following
\begin{defn}
A linear functional $F$ on $\cA$ is said to be \textit{positive} if $F(a^\ast a)\geq 0$ for all $a\in \cA$ and \textit{it strongly positive} if $F(a)\geq 0$ for all $a\in\cA_+.$
\end{defn}
Each strongly positive functional is positive, but Proposition \ref{notmomentfunctional} below shows that the converse is not true.
\begin{thm}\label{thm_hav}
A linear functional $F$ on $\cA$ is a $q$-moment functional if and only if $F$ is strongly positive.
\end{thm}
\begin{proof}
From the definition of the cone $\cA_+$ (Definition \ref{defn_pos}) it is obvious that $q$-moment functionals are strongly positive.

Suppose that $F$ is strongly positive. To prove that $F$ is a $q$-moment functional we need some preparations. First we define some auxiliary algebras.

Let $\cF$ be the $*$-algebra of all Borel functions $f(t)$ on $\dR_+$ which are polynomially bounded (that is, there exists a polynomial $p\in\dC[t]$ such that $|f(t)|\leq p(t)$ for $t\in \dR_+$). We denote by $\gX$ the $*$-algebra generated by an element $u$ and the $*$-algebra $\cF$ with defining relations
\begin{gather}\label{eq_gX_rel}
u^*u=uu^*=1,\ u f(t)=f(q^{1/2}t),\ f(t)u^*=f(q^{1/2}t),
\end{gather}
for $f\in\cF.$ Clearly, $\gX$ has a vector space basis $\{x^nc^{k};\ k\in\dN_0,\ n\in\dZ\},$ where $c^2=x^*x$ and $x^{-n}:=x^{*n}$ for $n<0,\ n\in\dZ.$ Hence there is an injective $*$-homomorphism $J$ of $\cA$ into $\gX$ given by $J(x)=uf_0,$ where $f_0(t)=t.$ We identify $J(a)$ and $a$ for $a\in\cA$ and consider $\cA$ as a $*$-subalgebra of $\gX.$ With a slight abuse of notation we shall write $x=ut$, where $t$ means the function $f_0(t)=t$ on $\dR_+$.

Let $\mu$ be a Borel measure on $\dR_+$ satisfying (\ref{eq_muD}). Then
\begin{gather*}
(\pi_\mu(f)\varphi)(t)=f(q^{1/2}t)\varphi(q^{1/2}t),\ (u\varphi)(t)=\varphi(q^{1/2}t),\\
\varphi\in\cD(\pi_\mu)=\set{\varphi\in L^2(\dR_+,\mu)|\ t^n\varphi\in L^2(\dR_+,\mu)\ \mbox{for all}\ n\in\dN},
\end{gather*}
defines a $*$-representation of $\gX$ on $\cH=L^2(\dR_+,d\mu)$ and $\overline{\pi_\mu(ut)}=X_\mu$ is the $q$-normal operator given by (\ref{eq_Amu}). Setting
\begin{gather*}
\gX_+:=\set{x\in\gX|\ \pi_\mu(x)\geq 0\ \mbox{for all measures }\ \mu\ \mbox{satisfying}\ (\ref{eq_muD})},
\end{gather*}
we clearly have $\cA_+=\gX_+\cap\cA.$

Let $\gX_b$ be the $*$-subalgebra of $\gX$ generated by $u$ and the subset $\cF_b$ of all $f\in \cF$ of compact support and consider the $*$-subalgebra $\cY=\cA+\gX_b$ of $\gX.$ Clearly, $\cA_+$ is cofinal in $\cY_+:=\cY\cap\cX,$ that is, for each $y\in\cY_+$ there exists $a\in\cA_+$ such that $a-y\in\cY_+.$ Therefore, since $\cA_+=\gX_+\cap\cA$, $F$ extends to a linear functional, denoted again by $F,$ such that $F(y)\geq 0$ for all $y\in\cY_+.$ Let $\pi_F$ denote the $*$-representation of $\cY$ with cyclic vector $\varphi$ obtained by the GNS construction from the functional $F$ (see e.g. \cite{s4}, Section 8.6). Then, by the GNS-construction,
\begin{gather}\label{gns}
F(y)=\langle\pi_F(y)\varphi,\varphi\rangle\ \mbox{for}\ y\in\cY.
\end{gather}

Let $\One_\Delta$ be the characteristic function of a Borel subset $\Delta\subseteq\dR_+$ and define $(E(\Delta)f)(t)=\overline{\pi_F(\One_\Delta)}f(t),\ f\in\cH(\pi_F).$ Then $E$ defines a spectral measure on $\dR_+.$ Let $U=\ov{\pi_F(u)}.$ From (\ref{eq_gX_rel}) it follows that $UE(\Delta)U^*=E(q^{-1/2}\Delta).$ Let $C=\int_0^\infty\lambda dE(\lambda).$ Then $X:=UC$ is a $q$-normal operator on $\cH(\pi_F)$ by Proposition \ref{prop_char}.
The proof is complete once we have shown that $\pi_F(x)\subseteq X$, or equivalently,
\begin{gather}\label{eq_b0}
\pi_F(a)\varphi\in\cD(X)~~{\rm and}~~ \pi_F(x)\pi_F(a)\varphi=X\pi_F(a)\varphi~~{\rm for}~~a\in \cA.
\end{gather}
Indeed, because $X$ is $q$-normal, by Lemma \ref{lemma_dense} and Definition \ref{defn_wellbeh} there is a well-behaved $\ast$-representation $\pi$ of $\cA$ on $\cD(\pi):=\cap_n \cD(X^n)$ such that $\pi(x)=X\lceil \cD(\pi).$ The relation $\pi_F(x)\subseteq X$ implies that $\pi_F\subseteq \pi$. Therefore, by (\ref{gns}), $F(a)=\langle\pi_F(a)\varphi,\varphi\rangle=\langle\pi(a)\varphi,\varphi\rangle$ for $a\in\cA,$ so $F$ is a $q$-moment functional.

 Let $f\in\cF_b$ and $k\in \dZ$. Then the operator $\pi_F(f)$ is bounded and we have $f(C)=\int_0^\infty f(\lambda)d E(\lambda)=\ov{\pi_F(f)}$ by the spectral calculus. Therefore, since $u^kf(t)=f(q^{k/2}t)u^k$ by (\ref{eq_gX_rel}), we have
\begin{align}\label{eq_b1}
C\pi_F(u^kf(t))& =C\pi_F(f(g^{k/2}t))\pi_F(u^k)=Cf(g^{k/2}t)\pi_F(u^k)\nonumber\\ &=\pi_F(tf(g^{k/2}t))\pi_F(u^k)=\pi_F(tu^kf(t)).
\end{align}
We prove (\ref{eq_b0}) for $a=u^{\tau n}t^n,$ where $\tau=\pm 1$ and $n\in\dN_0$. Let $\vre>0$ be fixed. We choose $\alpha_\vre >0$ such that $t^{2n}\leq \vre (1+t^{2n+2})$ for $t>\alpha_\vre$ and denote the characteristic function of the interval $[0,\alpha_\vre]$ by $\One_\vre.$ Setting $g_\vre(t)=\One_\vre(t)t^n$ and $f_\vre(t)=t^n-g_\vre(t),$ we have $g_\vre\in\cF_b$ and $f_\vre(t)^2\leq\vre(1+t^{2n+2})$ for $t\in\dR_+.$
Hence
\begin{gather}\label{eq_b2}
\vre(t^2+t^{2n+4})-t^2f_\vre(t)^2\in\cY_+,~~~\vre(1+t^{2n+2})-f_\vre(t)^2\in\cY_+.
\end{gather}
Now we compute
\begin{gather}\label{est1}
\nonumber\|\pi_F(x)\pi_F(u^{\tau n}t^n)\varphi-X \pi_F(u^{\tau n}g_\vre)\varphi\|^2=\|\pi_F(ut)\pi_F(u^{\tau n}t^n)\varphi -\pi_F(u)C \pi_F(u^{\tau n}g_\vre)\varphi\|^2=\\
\nonumber=\|\pi_F(u)(\pi_F(tu^{\tau n}t^n)\varphi -\pi_F(tu^{\tau n}g_\vre)) \varphi\|^2=\|\pi_F(q^{-n/2}u^{\tau(n+1)})t(t^n- g_\vre(t)))\varphi\|^2=\\
=q^{-n}\norm{\pi_F(t f_\vre(t))\varphi}^2=q^{-n}\langle \pi_F(t^2f_\vre(t)^2)\varphi,\varphi\rangle = q^{-n}F(t^2 f_\vre(t)^2) \leq\vre q^{-n}F(t^2+t^{2n+4}).
\end{gather}
Here we used first equations (\ref{eq_b1}) and (\ref{eq_gX_rel}), then the fact that $\pi_F(u^{\tau (n+1)})$ preserves the norm and equation (\ref{gns}) for $y=t^2f_\vre(t)^2$. Since $F$ is $\cY_+$-positive, we have $F(\vre(t^2+t^{2n+4})-t^2f_\vre(t)^2)\geq 0$ by (\ref{eq_b2}) which gives the inequality in the last line.

Using now the fact that $\vre(1+t^{2n+2})-f_\vre(t)^2\in\cY_+$ by (\ref{eq_b1}) we derive
\begin{gather}\label{est2}
\|\pi_F(u^{\tau n}t^n)\varphi-\pi_F(u^{\tau n}g_\vre)\varphi\|^2 =
\|\pi_F(u^{\tau n}f_\vre)\varphi\|^2\nonumber\\  =\|\pi_F(f_\vre(t)\varphi\|^2= F(f_\vre(t)^2) \leq\vre F(1+t^{2n+2}).
\end{gather}
Letting $\vre\to 0$, (\ref{est1}) and (\ref{est2}) imply that
$$
\pi_F(u^{\tau n}g_\vre )\varphi\to\pi_F(u^{\tau n} t^n)\varphi\ \ \mbox{and}\ \ X\pi_F(u^{\tau n}g_\vre)\varphi\to\pi_F(x)\pi_F(u^{\tau n} t^n)\varphi.
$$
Therefore, since $X$ is closed, we have $\pi_F(x)\pi_F(u^{\tau nt^n})\varphi\in \cD(X)$ and $\pi_F(x)\pi_F(u^{\tau nt^n})\varphi=X\pi_F(u^{\tau n}t^n)\varphi.$ This proves (\ref{eq_b0}) for $a=u^{\tau n}t^n$. Since these elements span $\cA$, (\ref{eq_b0}) holds for all $a\in \cA$ which completes the proof.
\end{proof}

\begin{thm}\label{notmomentfunctional}
There exists a positive linear functional on $\cA$ which is not a $q$-moment functional.
\end{thm}
\noindent Before we prove this we state two technical lemmas. The first one is taken from \cite{s3}, Lemma \nolinebreak 2.
\begin{lemma}\label{lemma_closcone}
Let $\cA$ be a unital $\ast$-algebra which has a faithful $\ast$-re\-pre\-sen\-ta\-ti\-on $\pi$ and is the union of a sequence of finite dimensional subspaces $E_n$, $n \in \dN$. Assume that for each $n \in \dN$ there exists a number $k_n\in \dN$ such that the following is satisfied: if $a \in \sum \cA^2$ is in $E_n$, then we can write $a$ as a finite sum $\sum_j a_j^*a_j^{}$ such that all $a_j$ are in $E_{k_n}$.\\ Then the cone $\sum \cA^2$ is closed in $\cA$ with respect to the finest locally convex topology on $\cA$.
\end{lemma}

\begin{lemma}\label{prop_faith}
Suppose that $\pi$ is a well-behaved representation such that $\pi(x)\not\equiv 0.$ Then $\pi$ is faithful.
\end{lemma}
\begin{proof}
Since $\pi$ is well-behaved, there is a $q$-normal operator $X$ such that $\pi(x)=X\upharpoonright\cD^\infty(X).$ By Theorem \ref{thm_spec}, $X$ is a direct sum of operators $X_{\mu_i}$. Since $\pi(x)\neq 0$, $X_{\mu_i}\neq 0$ for one $i$.

Suppose that $f(x,x^*)\in\cA,\ f\neq 0.$ It suffices to prove that there exists a vector $\varphi\in\cD^\infty(X_{\mu_i})$ such that $f(X_{\mu_i}^{},X_{\mu_i}^*)\varphi\neq 0.$

The polar decomposition $X_{\mu_i}=U_{\mu_i} C_{\mu_i}$ is given by (\ref{eq_Umu_Cmu}). From Proposition \ref{prop_char}$(iv)$ it follows that there are polynomials $f_k\in\dC[t]$, $k=-n,\dots,n$ sich that $f(X_{\mu_i}^{},X_{\mu_i}^*)=\sum_{k=-n}^{n}U_{\mu_i}^k f_k(C_{\mu_i}).$ Since $f\neq 0$, there is a $j$ such that $f_{j}\neq 0.$ Put $\varphi=\One_{q^{m/2}\Delta_q}(t)$ and choose $m\in\dZ$ such that the interval $q^{m/2}\Delta_q$ contains no zero of $f_{j}(t).$ Then, since ${\mu_i}\neq 0,$ we have $\mu_i(q^{m/2}\Delta_q)\neq 0$ by (\ref{eq_muD}) and hence $\varphi\neq 0.$ Using (\ref{eq_Umu_Cmu}) we calculate
\begin{gather*}
f(X_\mu^*,X_\mu)\varphi=\sum_{k=-n}^{n}U_\mu^k f_k(t)\One_{q^{m/2}\Delta_q}(t)=\sum_{k=-n}^{n}f_k(q^{k/2}t)\One_{q^{(m-k)/2}\Delta_q}(t).
\end{gather*}
If the latter would be zero, then $f_k(q^{k/2}t)\One_{q^{(m-k)/2}\Delta_q}(t)\equiv 0$ in $L^2(\dR_+,d\mu)$ for all $k,$ in particular for $k=j$ which is a contradiction. Thus $f(X_\mu^*,X_\mu)\varphi\neq 0.$
\end{proof}
\noindent\textit{Proof of Theorem \ref{notmomentfunctional}:}
We denote by $E_k$ the subspace of elements $f\in\cA,\deg f\leq k.$ Obviously, $\sum g_i^*g_i^{}\in E_{2k}$ implies that $g_i\in E_k$ for all $i$. By Lemma \ref{prop_faith} $\cA$ has a faithful representation. Therefore, Lemma \ref{lemma_closcone} applies, so the cone $\sum\cA(q)^2$ is closed in the finest locally convex topoloogy.

By Theorem \ref{thm_motzkin1} there exists an element $L\in\cA_+$ such that $L\notin\sum\cA^2.$ Since $\sum\cA^2$ is closed, by the separation theorem for convex sets there is a linear functional $F$ on $\cA$ such that $F(L)<0$ and $F(\sum\cA(q)^2)\geq 0.$ By the latter condition, $F$ is a positive linear functional. Since $F$ is not strongly positive (by $F(L)<0$), it is not a moment functional by Theorem \ref{thm_hav}.
\hfill $\Box$

\begin{defn}
A densely defined operator $X$ on a Hilbert space $H$ is a \emph{formally $q$-normal operator} if $\cD(X)\subseteq\cD(X^*)$ and $\norm{Xf}=\sqrt{q}\norm{X^*f}$ for $f\in\cD(X).$
\end{defn}

It is well-known \cite{cod} that there exist formally normal operators which have no normal extensions in larger Hilbert spaces. The next theorem shows that a similar result holds for formally $q$-normal operators.

\begin{thm}\label{prop_noext}
There exists a formally $q$-normal operator $X$ which has no $q$-normal extension in a possibly larger Hilbert space.
\end{thm}
\begin{proof}
We retain the notation from the proof of Theorem \ref{notmomentfunctional}. Let $\pi_F$ denote the GNS representation of $F$ with cyclic vector $\varphi$, see \cite{s4}. Then $F(a)=\langle\pi_F(a)\varphi,\varphi\rangle$ for $ a\in\cA$.

We show that $X:=\pi_F(x)$ is a formally $q$-normal operator which has no $q$-normal extension. Indeed, since $\pi_F$ is a $*$-representation of $\cA$, we have $\cD(X)=\cD(\pi_F)=\cD(\pi_F(x^*))\subseteq\cD(\pi_F(x)^*)=\cD(X^*).$ For $\psi\in\cD(X)$ we have
\begin{gather*}
\norm{X\psi}^2=\langle\pi_F(x)\psi,\pi_F(x)\psi\rangle=\langle \pi_F(x^*x)\psi,\psi\rangle=q^{-1}\langle \pi_F(xx^*)\psi,\psi\rangle=\\
=\langle X^*\psi,X^*\psi\rangle=q^{-1}\norm{X^*\psi}^2.
\end{gather*}

Assume that $Y$ is a $q$-normal operator on a (possible larger) Hilbert space such that $X\subseteq Y$. Then $L(X,X^*)\subseteq L(Y^{},Y^*)$ and hence
\begin{gather*}
\langle L(Y^{},Y^*)\varphi,\varphi\rangle=\langle L(X,X^*)\varphi,\varphi\rangle=\langle \pi_F(L)\varphi,\varphi\rangle=F(L)<0.
\end{gather*}
Since $L\in \cA_+$, this is a contradiction.
\end{proof}

\section{ A strict Positivstellensatz for $q$-polynomials}\label{strictpositivstellensatz}

The strict Positivstellensatz (Theorem \ref{thm_strpos}) proved in this section can be viewed as a $q$-analogue of the Reznick's Positivstellensatz \cite{rez}.

Let $f=\sum_{i,j}a_{ij}x^{*i}x^{j}\in\cA$ and $\deg f=m.$ We denote by $f_{m}=\sum_{\set{i+j=m}}a_{ij}x^{*i}x^{j}$ the highest order degree part of $f.$ We write $f_{m}$ as
$$
f_m=\sum_{r=0}^{\lfloor m/2\rfloor}b_r(x^*x)^rx^{m-2r}+\sum_{r=0}^{\lfloor m/2\rfloor-1}b_{-r}x^{*(m-2r)}(x^*x)^r,\ b_r\in\dC.
$$
The symbol of $f$ is the function $\sigma_f(\omega,\overline{\omega})$ on $\dC\setminus\set{0}$ defined by
\begin{align}\label{eq_symbol}
\sigma_f(\omega,\overline{\omega}):=\sum_{r=0}^{\lfloor m/2\rfloor}b_r|\omega|^{-r(m-2r)}\omega^{\frac{m}2-r}+
\sum_{r=0}^{\lfloor m/2\rfloor-1}b_{-r}\overline{\omega}^{\frac{m}2-r}|\omega|^{-r(m-2r)}.
\end{align}
Let $\cN$ denote the set consisting of $\Un$ and all finite products of elements $q^kx^*x+1,$ where $k\in\dZ.$

\begin{thm}\label{thm_strpos}
Let $f=f^*\in\cA$, $\deg f=4m,\ m\in\dN$. Suppose that:
\begin{enumerate}
	\item[$(i)$] For every $q$-normal operator $X$ there exists a $\varepsilon_X>0$ such that
        $$\langle f(X,X^*)\varphi, \varphi\rangle\geq \varepsilon_X\langle\varphi,\varphi\rangle,\ \varphi\in\cD^\infty(X),$$
	\item[$(ii)$]  $\sigma_f(\omega,\overline{\omega})>0$ for all $\omega\in\dT:=\set{z\in \dC:|z|=q^{1/2}}.$
\end{enumerate}
Then there exists an element $b\in\cN$ such that $b^*fb\in\sum\cA(q)^2.$
\end{thm}

The proof of this theorem follows a similar pattern as the proof of the strict Positivstellensatz for the Weyl algebra given in \cite{s1}.
We first recall a basic definition and a result from \cite{s1}, see e.g. \cite{s2}.

A unital $\ast$-algebra $\cY$ is called {\it algebraically bounded} if for each element $y\in \cB$ there exists a $\lambda_y>0$ such that
\begin{align}\label{algbounded}
\lambda_y \cdot \Un - y^\ast y \in \sum \cY^2.
\end{align}
\begin{lemma}\label{prop_strpossatz}
Let $\cY$ be an algebraically bounded $*$-algebra and $y=y^*\in\cB$. If
\begin{align}\label{strictcondition}
\langle\pi(y)\varphi,\varphi\rangle>0 ~~{\rm for~all}~~ \varphi\in\cH_\pi,\ \varphi\neq 0,
\end{align}
for $*$-representation $\pi$ of $\cY$, then $y\in\sum\cY^2.$
\end{lemma}
\begin{proof}
Assume to the contrary that $y\notin \sum \cX^2$. Since $\cY$ is algebraically bounded, $\Un$ is an internal point of the wedge $\sum \cY^2$. Therefore, by the Eidelheit separation theorem for convex sets \cite{koethe}, there exists a linear functional $F$ on $\cY$ such that $F(y) \leq 0$, $F(\Un)> 0$, and $F(\sum\cY^2)\geq 0$. If $\pi_F$ denotes the GNS-representation of $F$ with cyclic vector $\varphi$, then $F(y)=\langle \pi_F(y)\varphi,\varphi\rangle \leq 0$. Since $F(\Un)=\|\varphi\|^2> 0$, the latter contradicts (\ref{strictcondition}).
\end{proof}

\medskip
The proof of the theorem will be divided into three steps.

\medskip
\noindent\textbf{I.} Let $\rho$ be a fixed well-behaved representation of $\cA$ such that $\rho(x)\neq 0.$ By Lemma \ref{prop_faith}, $\rho$ is faithful. For notational simplicity we identify $a\in \cA$ with $\rho(a)$. Then $\ov{x}$ is a $q$-normal operator and $\cA$ becomes a $*$-algebra of operators acting on the invariant dense domain $\cD(\rho)=\cD^{\infty}(\overline{x}).$ Define the following operators
$$y_k:=x^2(q^kx^*\ov{x}+1)^{-1},\ v_k:=x(q^kx^*\ov{x}+1)^{-1},\ z_k:=(q^kx^*\ov{x}+1)^{-1},\ k\in\dZ.$$
Is is easily checked that $y_k^{},\ y_k^*,\ v_k^{},\ v_k^*$ are bounded operators which map the domain $\cD^\infty(\overline{x})$ into itself. Let $\gX$ be the $*$-algebra of operators on $\cD^{\infty}(\overline{x})$ generated by $x,x^*,y_k^{},y_k^*,v_k^{},v_k^*,z_k,\ k\in\dZ.$ Then the follwing relations hold in $\gX:$
\begin{align}
\label{eq_a1}xz_k^{}=v_k^{},\ z_k^{}x^*=v_k^*,\ x^2z_k^{}=y_k^{},\ z_k^{}x^{*2}=y_k^*\\
\label{eq_a2}xz_k=z_{k+1}x,\ x^*z_k=z_{k-1}x^*\\
\label{eq_a3}z_kz_m=z_mz_k,\ \ z_k^*=z_k^{},\\
\label{eq_a4}q^{2k+1}y_k^*y_k^{}+q^kv_k^*v_k^{}+z_k=1,\ \ \ q^kv_k^*v_k^{}+z_k^2=z_k,\\
\label{eq_a5}y_m^*y_k^{}=y_k^*y_m^{},\ y_k^*y_k^{}=q^{-4}y_{k-2}^{}y_{k-2}^*,\ v_k^{}v_k^*=qv_{k+1}^*v_{k+1}^{},\\
\label{eq_a6}y_kz_m=z_{m+2}y_k,\ \ \ y_k^*z_m^{}=z_{m-2}^{}y_k^*,\\
\label{eq_a7}v_kz_m=z_{m+1}v_k,\ \ \ v_k^*z_m^{}=z_{m-1}^{}v_k^*,\\
\label{eq_a8}y_my_k^*=qy_{k+1}^*y_{m+1}^{},\ \ \ v_kz_m=v_mz_k,\\
\label{eq_a9}y_ky_m=y_{m+2}y_{k-2},\ y_m^*y_k^*=y_{k-2}^*y_{m+2}^*,\\
\label{eq_a10}q^kz_k(1-z_m)=q^mz_m(1-z_k),\\
\label{eq_a11}q^{k+m+1}y_k^*y_m^{}=(1-z_k)(1-z_m),\ k,m\in\dZ.
\end{align}
Let $\gY$ denote the subalgebra of $\gX$ generated by $\Un{},y_k^{},y_k^*,v_k^{},v_k^*,z_k,$ where $k\in\dZ.$ From (\ref{eq_a4}) it follows that condition (\ref{algbounded}) holds for the algebra generators $y=y_k^{},y_k^*,v_k^{},v_k^*,z_k$ of $\cY$. Hence $\gY$ is an algebraically bounded $\ast$-algebra by Lemma 2.1 in \cite{s1}. \\ Our next aim is to study representations of $\gY.$

\medskip
\noindent\textbf{II.} Suppose that $\pi$ is a non-zero $*$-representation of $\gY$ on a Hilbert space $\cH_\pi.$ Since $\gY$ is algebraically bounded, all operators $\pi(y)$, $y\in \cY$, are bounded, so we can assume that $\cD(\pi)=\cH_\pi.$ Let $\cH_0=\pi(z_0),\ \cH_1=\ker~\pi(\Un{-}z_0).$ By (\ref{eq_a10}) we have $\ker\pi(z_k)=\cH_0$ and $ \ker\pi(\Un{-}z_k)=\cH_1,\ k\in\dZ.$
From (\ref{eq_a4}) it follows that $\pi(v_k)\upharpoonright\cH_0=0$. The third relation in (\ref{eq_a5}) yields $\pi(v_k^*)\upharpoonright\cH_0=0.$ Further, (\ref{eq_a7}) implies that $\cH_0$ is invariant under $\pi(y_k^{}),\pi(y_k^*),\ k\in\dZ.$ It follows from (\ref{eq_a4}) that $\pi(v_k^{}),\pi(v_k^*),\pi(y_k^{}),\pi(y_k^*)$ restricted onto $\cH_1$ are $0.$ The preceding implies that $\cH_0$ and $\cH_1$ are invariant subspaces of the representation $\pi.$ Let $\pi=\pi_0\oplus\pi_1\oplus\pi_2$ be the corresponding decomposition of $\pi$ on $\cH_\pi=\cH_0\oplus\cH_1\oplus\cH_2.$

\medskip
Now we analyze the three subrepresentations $\pi_0,\pi_1$ and $\pi_2.$ We begin with $\pi_0$. By construction of $\pi_0$ we have $\pi_0(z_k)=\pi_0(v_k^{})=\pi_0(v_k^*)=0$ for
$ k\in\dZ.$ From (\ref{eq_a4}) we obtain $\pi_0(y_k^*y_k^{})=1/q^{2k+1}.$ By (\ref{eq_a5}), $\pi_0(y_k^{}y_k^*)=q^4\pi_0(y_{k+2}^*y_{k+2}^{})=1/q^{2k+1},$ so that $\pi_0(q^{k+1/2}y_k)$ is unitary. Finally, it follows from (\ref{eq_a11}) that all operators $\pi_0(q^{k+1/2}y_k)$ coincide. Hence there exists a unitary operator $Y$ on $\cH_0$ such that
\begin{gather}\label{eq_Y}
\pi_0(y_k)=q^{-k-\frac{1}2}Y,\ k\in\dZ.
\end{gather}

Next we consider $\pi_1$. As noted above, $\pi_1(\Un)=\pi_1(z_k)=I_{\cH_1}$ and $\pi_1(v_k^{})=\pi_1(v_k^*)=\pi_1(y_k^{})=\pi_1(y_k^*)=0.$

Finally, we turn to $\pi_2.$ It is convenient to introduce the notation
$$
Z_k:=\pi_2(z_k),\ V_k:=\pi_2(v_k),\ Y_k:=\pi_2(y_k).
$$
Then $\ker Z_k=\ker(I-Z_k)=\set{0}$ by the construction of $\pi_2$. Combined with (\ref{eq_a4}) we conclude that $0<Z_k<I.$ Further, (\ref{eq_a11}) implies that all operators $q^{-k}(Z_k^{-1}-1),k\in\dZ,$ are equal and positive. Set $C:=q^{-k/2}(Z_k^{-1}-1)^{1/2}.$ Then $Z_k=q^kC^2+1)^{-1}$ and using (\ref{eq_a4}) we get $|V_k|=C(q^kC^2+1)^{-1}.$ Let $V_k=U_k|V_k|$ be the polar decomposition of $V_k.$ Note that $\ker C=\set{0}.$ From (\ref{eq_a8}) we derive
$$U_kC(q^kC^2+1)^{-1}(q^mC^2+1)^{-1}=U_mC(q^mC^2+1)^{-1}(q^kC^2+1)^{-1},\ k,m\in\dZ.$$
This implies that all operators $U_k,\ k\in\dZ,$ are equal. Set $U:=U_k$. Since $\ker V_k^{}=\ker V_k^*=\set{0},$ $U$ is unitary. By (\ref{eq_a7}) we have
$$UC(q^kC^2+1)^{-1}(q^mC^2+1)^{-1}=(q^{m+1}C^2+1)^{-1}UC(q^kC^2+1)^{-1},\ k,m\in\dZ,$$
and since $C(q^kC^2+1)^{-1}$ is invertible,
$$
U(q^mC^2+1)^{-1}=(q^{m+1}C^2+1)^{-1}U.
$$
The latter is equivalent to $UC^2U^\ast=qC^2.$ Therefore, by Proposition \ref{prop_char}, $X:=UC=V_kZ_k^{-1}$ is a $q$-normal operator. Further, $\pi_2$ leaves $\cD^\infty(X)=\cD^\infty(C)$ invariant. Indeed, by construction this is true for the generators and hence for all elements of $\cY$.

Let $h\in\gX.$ Using the relations (\ref{eq_a2}) and $(q^kx^*x+1)z_k=\Un_\gX$ it follows that $h$ is of the form $h=h_1z_{k_1}z_{k_2}\dots z_{k_m}$, where $h_1\in\cA$ and $ k_1,\dots,k_m\in\dZ.$ Since the operators $z_{k_1},\dots,z_{k_m}$ map $\cD^\infty(\overline{x})$ bijectively onto itself, $h=0$ if and only if $h_1=0.$ Therefore, the $\ast$-representation $\pi_2$ gives rise to a unique $\ast$-representation $\widetilde{\pi}_2$ of $\cX$ on $\cD^{\infty}(X)$ defined by $\widetilde{\pi}_2(y)= \pi_2(y)\upharpoonright\cD^\infty(X)$ for $y\in \cY$,
\begin{align*}
\widetilde{\pi}_2(x)= X\upharpoonright\cD^\infty(X),~{\rm and}~~ \widetilde{\pi}_2(z_k) =(q^kX^*X+1)^{-1}\upharpoonright\cD^\infty(X),~ k\in\dZ.
\end{align*}

\medskip
\noindent\textbf{III.} Now let
$f$ be as in Theorem \ref{thm_strpos} and let $f_{4m}=\sum_{i+j=4m}a_{ij}x^{*i}x^{j}$ be its highest degree part. From (\ref{eq_a1}) and (\ref{eq_a2}) it follows that $$y:=z_0^mf(x,x^*)z_0^m\in\gY.$$
Our next aim is to apply Lemma \ref{prop_strpossatz} in order to conclude that $y\in \sum \cY^2.$

Let $\pi=\pi_0\oplus\pi_1\oplus\pi_2$ be a representation of $\gY$ as analyzed above.

\medskip
First we determine $\pi_0(y).$
Suppose $i,j\in \dN_0$ and $i+j<4m$. Applying the relations (\ref{eq_a1}) and (\ref{eq_a2}) it follows that $z_0^mx^{*i}x^{j}z_0^m=w_1w_2\dots w_s$,
where each $w_l$ is equal to one of the elements $y_k^{},y_k^*,v_k^{},v_k^*, z_0.$ Since $i+j=4m$, not all elements $w_l$ can be equal to some $y_k.$ Therefore, since $\pi_0(z_k)=\pi_0(v_k)=\pi_0(v_k^*)=0,$ we obtain $\pi_0(z_0^mx^{*i}x^{j}z_0^m)=0.$
Hence $\pi_0(y)=\pi_0(z_0^mfz_0^m)=\pi_0(z_0^mf_{4m}z_0^m).$
Now we write
\begin{gather}\label{eq_f4m}
f_{4m}=\sum_{k=0}^{2m}b_k(x^*x)^{2m-k}x^{2k}+\sum_{k=1}^{2m}b_{-k}x^{*2k}(x^*x)^{2m-k}.
\end{gather}
For the monomial $(x^*x)^{2m-k}x^{2k}$ we treat the cases $k\leq m$ and $k>m.$
First suppose that $k\leq m.$ Using relations (\ref{eq_a1})-(\ref{eq_a3}) we compute
\begin{gather*}
z_0^m(x^*x)^{2m-k}x^{2k}z_0^m=z_0^{m}(x^*x)^{2m-k}z_{2k}^{m-k}x^{2k}z_0^{k}=\\
=(z_0 x^*x)^{m} (x^*x z_{2k})^{m-k}x^{2k}z_0^{k}=\\
=(\Un-z_0)^{m}(q^{-2k}(\Un-z_{2k}))^{m-k}(x^2 z_{2k-2})(x^2 z_{2k-4})\dots (x^2 z_0)=\\
=(\Un-z_0)^{m}(q^{-2k}(\Un-z_{2k}))^{m-k}y_{2k-2}y_{2k-4}\dots y_0.
\end{gather*}
Applying $\pi_0$ to both sides and using (\ref{eq_Y}) we derive
\begin{gather*}
\pi_0(z_0^m(x^*x)^{2m-k}x^{2k}z_0^m)=q^{-2k(m-k)}\pi_0(y_{2k-2}y_{2k-4}\dots y_0)=\\
=q^{-2k(m-k)}q^{-k^2+k/2}Y^k=|q^{1/2}|^{-2k(2m-k)}(q^{1/2}Y)^k.
\end{gather*}
In the case $k>m$ we obtain
\begin{gather*}
z_0^m(x^*x)^{2m-k}x^{2k}z_0^m=z_0^{2m-k}(x^*x)^{2m-k}z_0^{k-m}x^{2k}z_0^{m}=\\
=(z_0 x^*x )^{2m-k}\left(z_0^{k-m}x^{2(k-m)}\right)\left(x^{2m}z_0^{m}\right)=\\
=(\Un-z_0)^{2m-k}(x^2 z_{-2})(x^2 z_{-4})\dots (x^2 z_{-2(k-m)})\times\\
\times(x^2 z_{2(m-1)})(x^2 z_{2(m-2)})\dots (x^2 z_{0})=\\
=(\Un-z_0)^{2m-k}y_{-2}y_{-4}\dots y_{-2(k-m)}\cdot y_{2(m-1)}y_{2(m-2)}\dots y_{0}.
\end{gather*}
Again, applying $\pi_0$ and using (\ref{eq_Y}) we get
\begin{gather*}
\pi_0(z_0^m(x^*x)^{2m-k}x^{2k}z_0^m)=\\
=\pi_0(y_{-2}y_{-4}\dots y_{-2(k-m)})\pi_0(y_{2(m-1)}y_{2(m-2)}\dots y_{0})=\\
=q^{(k-m)(k-m+1)-\frac{k-m}2}Y^{k-m}q^{-m(m-1)-\frac{m}2}Y^m=|q^{1/2}|^{-2k(2m-k)}(q^{1/2}Y)^k.
\end{gather*}
Proceeding in a similar manner we derive
$$\pi_0(z_0^m x^{*2k}(x^*x)^{2m-k}z_0^m)=|q^{1/2}|^{-2k(2m-k)}(q^{1/2}Y^*)^k.$$
Let $Y=\int_\dT \omega~dE(\omega)$ be the spectral decomposition of the unitary operator $Y$. Comparing the preceding computations with the definition of $\sigma_f$ we get $\pi_0(y)= \int_\dT \sigma_f(q^{1/2}\omega,q^{1/2}\overline{\omega}) ~dE(\omega).$
From assumption (ii) it follows that there exists $\vre >0$ such that $\sigma_f(q^{1/2}\omega,q^{1/2}\overline{\omega})\geq \vre $ for $\omega \in \dT$. Hence
\begin{align*}
\langle \pi_0(y)\psi,\psi \rangle = \int_\dT \sigma_f(q^{1/2}\omega,q^{1/2}\overline{\omega})~d\langle E(\omega)\psi,\psi\rangle \geq \vre \|\psi\|^2 ,~~\psi \in\cH_0.
\end{align*}

\medskip
Applying assumption (i) to the $q$-normal operator $X=0$ yields $a_{00}=f(0,0)>0.$ By (\ref{eq_a1}) and (\ref{eq_a2}) we have $\pi_1(y)=\pi(a_{00}z_0^{2m})=a_{00}\cdot I.$

\medskip
Finally we turn to $\pi_2.$ As shown above, there exists a $*$-representation $\widetilde{\pi}_2$ of the $\ast$-algebra $\gX$ on $\cD^\infty(X)$ such that $\widetilde{\pi}_2\upharpoonright\cA$ is well-behaved and $\widetilde{\pi}_2(y)=\pi_2(y)\upharpoonright\cD^\infty(X)$ for $y\in \cY.$
Using assumption (i) of the theorem we obtain for $\zeta\in\cD^\infty(X),$
\begin{gather*}
\langle\pi_2(y)\zeta,\zeta\rangle=\langle\pi_2(z_0^mfz_0^m)\zeta,\zeta\rangle=\langle\widetilde{\pi}_2(z_0^mfz_0^m)\zeta,\zeta\rangle=\\
=\langle\widetilde{\pi}_2(f)\widetilde{\pi}_2(z_0^m)\zeta,\widetilde{\pi}_2(z_0^m)\zeta\rangle\geq\varepsilon_X\norm{\widetilde{\pi}_2(z_0^m)\zeta}^2=\varepsilon_X\norm{\pi_2(z_0^m)\zeta}^2.
\end{gather*}
Since $\pi_2(y)$ and $\pi_2(z_0)$ are bounded operators and $\ker~\pi_2(z_0)=\{0\}$ by construction, we conclude that $\langle\pi_2(y)\zeta,\zeta\rangle>0$ for all $\zeta\in\cH_2,$ $\zeta\neq 0.$

Since $\pi=\pi_0\oplus\pi_1\oplus\pi_2$, it follows from the preceding analysis that $\langle \pi(y)\psi,\psi\rangle >0$ for all $\psi\in \cH_\pi$, $\psi\neq 0$. Therefore, by Lemma \ref{prop_strpossatz},
$$g=\sum_{i=1}^rg_i^*g_i^{}\in\sum\gY^2.$$
The relations (\ref{eq_a2}) imply that in the algebra $\cX$ each $g_i\in \cY$ can be written $g_i=f_ih_i,$ where $f_i\in\cA$ and $h_i$ is a finite product of elements $z_j$. That is, $h_i^{-1}\in\cN\subseteq\cA.$ Multiplying both sides of the equation
$$
g=z_0^mfz_0^m=\sum_{i=1}^r(f_ih_i)^*f_ih_i
$$
by $(h_1 h_2\dots h_r z_0^m)^{-1}$ from the left and from the right we obtain
$$
bfb=\sum_{i=1}^r\widetilde{f}_i^*\widetilde{f}_i^{}\in\cA(q)^2,
$$
where $\widetilde{f}_i=f_ih_i (h_1 h_2\dots h_r z_0^m)^{-1}\in\cA$ and $b=(h_1 h_2\dots h_r)^{-1}\in\cN.$

\hfill$\Box$

The next example illustrates the assertion of the strict Positivstellensatz. Its proof is analogous to the proof of Theorem \ref{thm_motzkin1}.

\begin{prop}\label{prop_motzkin_str}
Suppose that $q=1/2.$ Then:
\begin{enumerate}
\item[$(i)$] $L:=(x x^*)^2-(x+x^*)^2+3.7\notin\sum \cA^2,$
\item[$(iii)$] $L$ satisfies the assumptions $(i)$ and $(ii)$ in Theorem \ref{thm_strpos},
\item[$(ii)$] $(1+qx^*x)L(1+qx^*x)\in \sum \cA^2.$
\end{enumerate}
\end{prop}

\bibliographystyle{amsalpha}

\end{document}